\documentclass[12pt,reqno]{amsart}
\usepackage{amssymb}
\usepackage{amsthm}
\usepackage{bm}
\usepackage{latexsym}
\usepackage{amsmath}
\usepackage{mathrsfs}
\usepackage{caption}
\usepackage{amscd}
\usepackage[all,cmtip]{xy}
\usepackage[usenames]{color}
\usepackage{graphicx}
\usepackage{color}
\usepackage{amscd}
\usepackage{float}
\usepackage{graphics}
\usepackage{tikz}
\usepackage{tikz-cd}
\usepackage{comment}
\usepackage{xspace}
\usepackage{mathtools}

\usepackage{amssymb}
\usepackage{amsthm}
\usepackage{graphicx}
\usepackage{subfig}
\usepackage{enumerate}
\usepackage{hyperref}
\usepackage{multirow}

\usetikzlibrary{patterns,decorations.pathreplacing}
\usepackage{marginnote}

\theoremstyle{plain}

\newtheorem{theorem}{Theorem}[section]
\newtheorem{proposition}[theorem]{Proposition}
\newtheorem{lemma}[theorem]{Lemma}
\newtheorem{corollary}[theorem]{Corollary}

\newtheorem{question}[theorem]{Question}

\theoremstyle{definition}

\newcommand{\appsection}[1]{\let\oldthesection\thesection
\renewcommand{\thesection}{Appendix \oldthesection}
\section{#1}\let\thesection\oldthesection}

\newtheorem{definition}[theorem]{Definition}
\newtheorem{notation}[theorem]{Notation}

\theoremstyle{remark}

\newtheorem{claim}{Claim}
\newtheorem{remark}[theorem]{Remark}
\newtheorem{example}[theorem]{Example}

\DeclareMathOperator{\sing}{Sing}

\def\Z{{\mathbb{Z}}}

\def\Q{{\mathbb{Q}}}

\def\P{{\mathbb{P}}}

\begin{document}
\title{On accumulation points of volumes of stable surfaces with one cyclic quotient singularity}
\author[Diana Torres]{Diana Torres}
\email{dctorres1@uc.cl}
\address{Facultad de Matem\'aticas, Pontificia Universidad Cat\'olica de Chile, Campus San Joaqu\'in, Avenida Vicu\~na Mackenna 4860, Santiago, Chile.}


\begin{abstract} The set of volumes of stable surfaces does have accumulation points. In this paper, we study this phenomenon for surfaces with one cyclic quotient singularity, towards answering the question under which conditions we can still have boundedness. Effective bounds allow listing singularities that might appear on a stable surface after fixing its invariants. We find optimal inequalities for stable surfaces with one cyclic quotient singularity, which can be used to prove boundedness under certain conditions. We also introduce the notion of generalized T-singularity, which is a natural generalization of the well-known T-singularities. By using our inequalities, we show how the accumulation points of volumes of stable surfaces with one generalized T-singularity are formed.
\end{abstract}

\maketitle
\tableofcontents
\allowdisplaybreaks
\section{Introduction} \label{intro}

This paper is about studying the behavior of $K^2$ for complex stable surfaces with particular singularities. Stable surfaces are the surfaces used by Koll\'ar--Shepherd-Barron \cite{kollar1988threefolds} and Alexeev \cite{alexeev1994boundedness} to give a natural compactification to the moduli space of surfaces of general type. But the interest in them goes beyond that compactification. A current topic of study is the distribution of volumes $K^2$ in the set of positive rational numbers. A fundamental result is the \textit{Descending Chain Condition} (DCC for short) for $\{K^2\}$ (the set of all $K^2$ of stable surfaces), which is due to Alexeev. (Its most general version for log stable surfaces can be found in \cite{alexeev1994boundedness}.) In particular, the DCC property implies the existence of a minimum for $\{K^2\}$. It is still an open problem to know its value. Knowing the exact lower bound for $K^2$ can be used, for example, to explicitly bound the automorphism group for surfaces of general type. (See e.g. \cite{alexeev1994boundedness} and \cite{kollar1994log} for more motivation.) Various authors have found low values for $K^2$ (see e.g. \cite{blache1995example}, \cite{UY2017}, \cite{liu2017minimal}, \cite{alexeev2019log}, \cite{alexeev2016open}). On the other hand, upper bounds are not possible even if we fix the geometric genus \cite[Thm. 1.9]{UU2019}, contrary to what happens for smooth projective surfaces of general type.  

\vspace{0.3cm}

It turns out that we can also have accumulation points for the set of volumes of (log) stable surfaces. There has been a recent interest on understanding better the set of accumulation points $\textnormal{Acc}(\{K^2\})$, see e.g. \cite{kollar1994log}, \cite{blache1995example}, \cite{UY2017}, \cite{alexeev2019accumulation}, \cite{alexeev2016open}. In early times, Blache \cite{blache1995example} showed that $1\in \textnormal{Acc}(\{K^2\})$. It was done by constructing a family of stable surfaces with ten cyclic quotient singularities. Blache conjectured that $\mathbb{N}\subseteq \textnormal{Acc}(\{K^2\}) \subseteq \mathbb{Q}$. In \cite{alexeev2019accumulation} (see also \cite{UY2017}), it is shown that all natural numbers are accumulation points, and that iterated accumulation points can have arbitrary complexity in unbounded regions. Additionally, Alexeev and Liu \cite{alexeev2019accumulation} proved general results about volumes of log canonical surfaces, which has several implications. One of them is to solve the conjecture of Blache about the closure of $\{K^2\}$, which is indeed in $\Q$. Although it is not known if the set of $K^2$ is closed (for empty boundary). A full description of $\{K^2\}$ and $\textnormal{Acc}(\{K^2\})$ is still missing. 

\vspace{0.3cm}

This paper aims to describe how accumulation points of volumes of stable surfaces are formed, in the case of surfaces with only one cyclic quotient singularity. We find the following numerical constraints which optimally bound singularities\footnote{We refer to boundedness in this context to describe the possible cyclic quotient singularities which may occur on a surface with bounded invariants $K^2$ and $\chi$.} when we restrict to specific situations, such as T-singularities (recovering \cite[Thm. 1.1]{rana2019optimal} for example) or generalized T-singularities (see Lemma \ref{T generalizadas E.C=1}). Of course one cannot expect to bound all cyclic quotient singularities because of the existence of accumulation points, but this theorem gives a way to detect them. All notations will be introduced in the body of the paper.


\begin{theorem}\label{deltas}Let $W$ be a stable surface with only one cyclic quotient singularity of type $\frac{1}{n}(1,q)$ at $P \in W$. Let $$C=C_1+ \ldots +C_r$$ be the chain of the exceptional curves in the minimal resolution of $P$, and let $[b_1,\dots ,b_r]$ be its Hirzebruch-Jung continued fraction. Let $X$ be the minimal resolution of $W$, and let $\pi\colon X\to S$ be a minimal model of $X$. Then
\begin{equation}\label{sumc}
    \sum_{j=1}^{r}\big( b_j-2\big)\leq 2(K_W^2-K_S^2)+2\bigg(\frac{2(n-1)-q-q'}{n}\bigg)+ \delta -\pi^*K_S\cdot C,
\end{equation}
and 
\begin{equation}\label{r}
r\leq 13K_W^2-2K_S^2+38 -\bigg(\frac{2+q+q'}{n}\bigg)+\delta -\pi^*K_S\cdot C,
\end{equation}
where $0<q'<n$ with $qq'\equiv 1\ (\text{mod}\ n)$, and $\delta$ is the positive number computed in Lemma \ref{delta-bound} for distinct geometric situations.

\end{theorem}

We point out that the core of Theorem \ref{deltas} relies on finding explicit $\delta$'s and a classification of all possible geometric realizations. Bounding of $\delta$ for a sequence of stable surfaces with one cyclic quotient singularity is directly related to the existence of accumulation points. We will use Bogomolov-Miyaoka-Yau inequality for proving the bound in \eqref{r} in Theorem \ref{deltas}. However, the bound in \eqref{sumc}, and the computation of $\delta$ in Lemma \ref{delta-bound}, remain valid in any characteristic.

Theorem \ref{deltas} directly implies the following results about boundedness and accumulation points.

\begin{corollary}\label{Cor1} Let $c>0$, and let $\mathcal{S}$ be a set of stable surfaces $W$ with one cyclic quotient singularity, $K_S$ nef, and $K_W^2 \leq c$. Let $\sing(\mathcal{S})$ be the set of singularities of the surfaces in $\mathcal{S}$. Then $\sing(\mathcal{S})$ is finite if and only if the number of $2$'s at the extremes of every $[b_1,\ldots,b_r] \in \sing(\mathcal{S})$ is bounded.
\end{corollary}

\begin{corollary}\label{Cor2} Let $\{W_k\}$ be a sequence of stable surfaces with only one cyclic quotient singularity. Assume that for every $k$ the minimal model $S_k$ of $W_k$ has canonical class nef, and that $K_{W_k}^2 \leq c$ for a positive number $c$. If the cases (A),(B.2) or (D.3) in Lemma \ref{delta-bound} hold except for a finite number of indices $k$, then Acc$(\{K_{W_k}^2\})=\emptyset$.
\end{corollary}

Next, we introduce the set-up that will be used to define and work with generalized T-singularities.

\begin{definition}[see e.g. \cite{orlik1977algebraic}] Let $\{a_1,\dots,a_s\}$ be an ordered set of positive natural numbers. Let $p_{-1}=0$, $p_0=1$, $q_0=0$, $q_1=1$, and for $i\geq 1$, $$p_{i+1}=a_{i+1}p_i+p_{i-1} \ \ \ , \ \ \ q_{i+1}=a_{i+1}q_i+q_{i-1}.$$ We say that $\{a_1,\dots,a_s\}$ is \textit{admissible} if $p_i>0$ for $i=0,\dots ,s-1$.
\end{definition}

It is a straightforward calculation to show that if $\{a_1,\dots,a_s\}$ is admissible, then the Hirzebruch-Jung continued fraction $[a_1,\dots,a_s]$ is well-defined. 

\begin{definition}\label{admissible for chains}
Let $[b_1,\ldots,b_r]$ be a Hirzebruch-Jung continued fraction with $b_i\geq 2$ for all $i$. We say that $[b_1,\ldots,b_r]$ is \textit{admissible for chains} if
\begin{equation}\label{temp10}
    \{b_1,\dots,b_r,1,b_1,\dots,b_r,1,\dots,1,b_1,\dots,b_r\}
\end{equation}
 is admissible for any number of inserted $1$'s.
\end{definition}

As for regular Hirzebruch-Jung continued fractions, we think geometrically of $\{b_1,\dots,b_r,1,b_1,\dots,b_r,1,\dots,1,b_1,\dots,b_r\}$ as a chain of $\P^1$'s, where we have $(-1)$-curves inserted between some minimal resolution chains of the cyclic quotient singularity associated to $[b_1,\ldots,b_r]$. For example, we have that $[4]$ is admissible for chains, and it gives all the initial chains to construct all the T-singularities \cite[Prop.3.11]{kollar1988threefolds}.  We take this to define generalized T-singularities.

\begin{definition} Let $\{a_1,\ldots,a_s\}$ be an admissible set. Its \textit{reduced} Hirzebruch-Jung continued fraction is the continued fraction obtained after contracting all $(-1)$-curves in $\{a_1,\ldots,a_s\}$, and all the new $(-1)$-curves after that.
\end{definition}

\begin{notation}\label{reduced-notation} The reduced Hirzebruch- Jung continued fraction of $$\{b_1,\dots,b_r,1,b_1,\dots,b_r,1,\dots,1,b_1,\dots,b_r\},$$ where $u$ is the number of inserted $1$'s, will be denoted by $[b_1^u,\dots,b_{r_u}^u]$. Also, we will write $[b_1^0,\dots,b_{r_0}^0]$ to refer to $[b_1,\ldots,b_r]$. We write the singularity $[a_1,\ldots,a_s]$ to refer to the cyclic singularity associated to this continued fraction.
\end{notation}

\begin{definition}\label{generalized T-singularity} Let $[b_1,\dots,b_r]$ be a Hirzebruch-Jung continued fraction which is admissible for chains. We define the class of \textit{generalized T-singularity} of center $[b_1,\dots,b_r]$ inductively in the following way

\begin{itemize}
\item[(i)] The singularities $[b_1^u,\dots,b_{r_u}^u]$ for every $u\geq 0$ are generalized T-singularities. 

\item[(ii)] If $[a_1,\dots,a_s]$ is a generalized T-singularity, then so are $$[2,a_1,\dots,a_{s-1},a_s+1] \ \ \ \text{and} \ \ \ [a_1+1,a_2,\dots,a_s,2].$$

\item[(iii)]  Every generalized T-singularity of center $[b_1,\ldots,b_r]$ is obtained by starting with one of the singularities described in (i) and iterating the steps described in (ii).
\end{itemize}

We say that we apply the \textit{T-chain algorithm} if we apply iterations of (ii).
\end{definition}

It is clear that T-singularities are the generalized T-singularities of center $[4]$. Rana and Urz\'ua in \cite{rana2019optimal} showed an optimal bound of T-singularities for stable surfaces with one singularity. A natural question is whether that result remains valid for generalized T-singularities. We answer this question in the following theorem by describing how the accumulation points of $K^2$ on stable surfaces with one generalized T-singularity of fixed center are formed.

\begin{theorem}\label{A1} Let $\{W_k\}$ be a sequence of stable surfaces such that any $W_k$ has only one generalized T-singularity with a fixed center\\ $[b_1,\dots,b_{r}]$, say at $P_k \in W_k$. Suppose that the minimal model $S_k$ of the minimal resolution of $W_k$ has canonical class nef. Then $\{K_{W_k}^2\}$ has accumulation points if and only if $\{K_{W_k}^2\}$ satisfy the property (*) (see Definition \ref{property (*)}.)
\end{theorem}

It is shown in Proposition \ref{A3} that every accumulation point which is coming from a sequence as one described in Theorem \ref{A1}, can be constructed by blowing up a particular configuration of curves in a smooth surface and then contracting the new configuration obtained.

\subsubsection*{Acknowledgments} I am grateful to my advisor Giancarlo Urz\'ua for his guidance and support throughout this work. The results in this paper are part of my Ph.D. thesis at the Pontificia Universidad Cat\'olica de Chile. I would also like to thank S\"onke Rollenske for the hospitality during my stay at the Philipps-Universit\"at Marburg. Special thanks to Wenfei Liu, Julie Rana, and S\"onke Rollenske for many comments and suggestions. I was funded by the Agencia Nacional de Investigaci\'on y Desarrollo (ANID) through the beca DOCTORADO NACIONAL 2017/21171009.

\section{Preliminaries} 

In this section, we introduce the notation used throughout this article. The results listed below can be found in  \cite{fulton1993introduction}. We will only consider stable surfaces $W$ with one cyclic quotient singularity $P$. We denote its minimal resolution by $X$, and a minimal model of $X$ by $S$ (i.e. $S$ has no $(-1)$-curves).

A two dimensional cyclic quotient singularity is by definition the germ at the origin of the quotient of $\mathbb{C}^2$ by $\Z/n$. It is denoted by $\frac{1}{n}(1,q)$, where $\xi\cdot(x,y)\mapsto(\xi x,\xi^qy)$ is the action of $\Z/n$ on $\mathbb{C}^2$, $\xi$ is a primitive root of $1$, and gcd$(q,n)=1$. A cyclic quotient singularity can be constructed as a singularity of a toric surface. This construction can be used to obtain an explicit resolution of the singularity, which is entirely determined by the numbers $n$ and $q$ in the following way (see \cite[pp.31-50]{fulton1993introduction}).


\begin{proposition}\label{Resolution of P} Let $W$ be a surface with a singularity $\frac{1}{n}(1,q)$. Then, the minimal resolution $\phi\colon X \to W$ contains a chain $C$ of exceptional curves $C_1,\dots, C_r$ such that $C_j \simeq \mathbb{P}^1$, and

\begin{equation}
C_i\cdot C_j =\left\{ 
    \begin{array}{lcl}
     &1& \text{if } i=j\pm 1 \\
     & -b_j& \text{if } i=j\\
     &0& \text{otherwise}
    
    \end{array}
\right.
\end{equation}
where $[b_1,\dots , b_r]$ is the Hirzebruch-Jung continued fraction of $\frac{n}{q}$. We say that this singularity has length $r$.
\end{proposition}
Given the chain 
$C=C_1+ \cdots +C_r$, its dual graph is defined as in Figure \ref{The dual graph}, where the $i$-th vertex corresponds to the curve $C_i$, and the edge between the curves $C_j$ and $C_{j+1}$ corresponds to the point in the intersection between them.

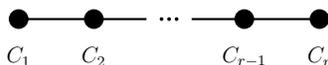
\begin{figure}[H]
\begin{center}
\begin{tikzpicture}[roundnode/.style={circle, draw=black, fill=white, thick,  scale=0.6},squarednode/.style={rectangle, draw=black, fill=white, thick, scale=0.7},roundnodefill/.style={circle, draw=black, fill=black, thick,  scale=0.6},letra/.style={rectangle, draw=white, fill=white, thick, scale=0.7}]

\draw[black, thick] (-2,0) -- (-0.3,0){};
\draw[black, thick] (0.3,0) -- (2,0){};

\node[roundnodefill] at (-2,0) {} ;
\node[letra] at (-2,-0.5) {$C_1$} ;
\node[roundnodefill] at (-1,0){};
\node[letra] at (-1,-0.5) {$C_2$} ;

\filldraw[black] (-0.1,0) circle (0.5pt) node[anchor=west] {};
\filldraw[black] (0,0)  circle (0.5pt) node[anchor=west] {};
\filldraw[black] (0.1,0) circle (0.5pt) node[anchor=west] {};

\node[roundnodefill] at (1,0){};
\node[letra] at (1,-0.5){$C_{r-1}$};
\node[roundnodefill] at (2,0){};
\node[letra] at (2,-0.5){$C_r$};

\end{tikzpicture}
\end{center}
\caption{The dual graph of $\frac{1}{n} (1,q)$.}
  \label{The dual graph}
\end{figure}

In this case, we have the following numerical equivalence 
\begin{equation}\label{eq2}
        K_X \equiv \phi^*K_W + \sum\limits_{j=1}^{r}a_jC_j
\end{equation} where the coefficients $a_j$ are rational numbers $a_j\in ]-1,0]$ called discrepancies. We will say that (\ref{eq2}) is the \textit{canonical class formula}.

\begin{remark}\label{triangular matrix} The vector of discrepancies is the solution of the following linear system  
\begin{center}
\begin{equation*}
A=
\left(\begin{array}{cccccc|c}

    -b_1&1&0&0&\cdots &0&b_1-2\\
    1&-b_2&1&0&\cdots &0&b_2-2\\
    0&1&\ddots&\ddots &\ddots&\vdots&\vdots\\
    \vdots&\ddots&\ddots&\ddots&1&0&b_{r-2}-2\\
    0&\cdots&0&1 &-b_{r-1}&1&b_{r-1}-2\\
    0&\cdots&0&0 &1&-b_r&b_r-2
\end{array}\right)
\end{equation*}
\end{center}

That linear system can be solved using the tridiagonal matrix algorithm because the matrix is diagonally dominant. Then, we obtain the discrepancies from the formulas $a_r=d_r$, and $a_j=d_j-c_ja_{j+1}$ for $j=2,\ldots r$, where $c_j,d_j$ are auxiliary coefficients defined as follows:
 \begin{itemize}
     \item $c_1=\dfrac{1}{-b_1}$, and $c_j=\dfrac{1}{-b_j-c_{j-1}}$ for $j=2,\ldots,r-1.$
     
     \item $d_1=\dfrac{b_1-2}{-b_1}$, and $d_j=\dfrac{1}{-b_j-c_{j-1}}$ for $j=2,\ldots,r.$
 \end{itemize}
\end{remark}


Following Proposition \ref{Resolution of P}, we denote by $[b_1,\dots,b_r]$ the continued fraction of $P$. Also, we denote by $q'$ the inverse of $q$ modulo $n$, that is, the unique integer $0<q'<n$ such that $qq'\equiv 1(mod\ n)$. 

\begin{proposition}\label{canonical divisor formula XW} Let $W$ be a normal projective surface with only one singularity and of type $\frac{1}{n}(1,q)$. Let $\phi\colon X \to W$ be the minimal resolution of $W$. Then we have $$K_X^2=K_W^2+\sum_{j=1}^{r}(2-b_j)+ \frac{2(n-1)-q-q'}{n}.$$
\end{proposition}
\begin{proof} See e.g. the proof of Proposition $3.4$ in \cite{urzua2010arrangements}.
\end{proof}


Now, let $\pi \colon X\to S$ be a birational morphism to the minimal model $S$. Thus it is a composition of blow ups, each of which contracts a single $(-1)$-curve $F_i \subset X_i$ to a point $x_{i-1}\in X_{i-1}$. In this way we have the diagram: 

\begin{equation*}
X=X_m\stackrel{\pi_m}{\to}X_{m-1}\stackrel{\pi_{m-1}}{\to}\cdots \stackrel{\pi_2}{\to}X_1\stackrel{\pi_1}{\to}X_0=S
\end{equation*}

Let us define $E_m:=F_m$, and for each $i \in \{1,\ldots,m-1\}$ 
\begin{equation}\label{Diagram}
E_i:=(\pi_{i+1}\circ \pi_{i+2} \circ \cdots \circ \pi_{m})^*(F_i)\subset X.
\end{equation}

It follows from the definition that $E_i^2=-1$ and $E_i\cdot E_j=0$ whenever $i\neq j$. Furthermore, we have that each $E_i$ is not necessarily reduced, and its support is a tree of smooth rational curves. Assuming that $m>0$, each $E_i$ contains at least one $(-1)$-curve, and their irreducible components intersect transversally at most once. Of course we have

\begin{equation}\label{canonical divisor formula WS} 
    K_W^2-K_S^2=\sum_{j=1}^{r}(b_j-2)-m-\bigg(\frac{2(n-1)-q-q'}{n}\bigg).
\end{equation}

\begin{lemma}\label{E.C} We have $\big( \sum_{i=1}^{m}E_i\big)\cdot C= \sum_{j=1}^{r}(b_j-2)-\lambda$, where $\lambda=\pi^*K_S\cdot C$.
\end{lemma}
\begin{proof} It follows directly from $K_X\cdot C={\sum\limits_{j=1}^{r}(b_j-2)}$, and $\sum_{i=1}^{m}E_i=K_X-\pi^*K_S$.
\end{proof}

In order to describe the behavior of the accumulations points of volumes, we will find a suitable lower bound for the intersection between $C$ and $\sum_{i=1}^{m}E_i$. We first introduce a graph $\Gamma_{E_i}$ for each exceptional divisor, as it was done in \cite[pp.9]{rana2017boundary}. It is constructed by replacing the $j$-th vertex in the dual graph of $C$, by a box if $C_i \subset E_i$. For instance, if we have $C_1,C_5$ belonging to $E_i$, the $\Gamma_{E_i}$ is as in Figure \ref{Diagram 1}.

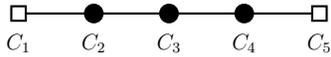
\begin{figure}[H]
\begin{center}
\begin{tikzpicture}[roundnode/.style={circle, draw=black, fill=white, thick,  scale=0.6},squarednode/.style={rectangle, draw=black, fill=white, thick, scale=0.7},roundnodefill/.style={circle, draw=black, fill=black, thick,  scale=0.6},letra/.style={rectangle, draw=white, fill=white, thick, scale=0.7}]

\draw[black, thick] (-2,0) -- (2,0);

\node[squarednode] at (-2,0){} ;
\node[letra] at (-2,-0.4) {$C_1$};

\node[roundnodefill] at (-1,0){} ;
\node[letra] at (-1,-0.4) {$C_2$} ;

\node[roundnodefill] at (0,0){};
\node[letra] at (0,-0.4){$C_3$};

\node[roundnodefill] at (1,0){};
\node[letra] at (1,-0.4){$C_4$} ;

\node[squarednode] at (2,0){};
\node[letra] at (2,-0.4) {$C_5$};
\end{tikzpicture}
\caption{Example of the graph of $E_i$.}
 \label{Diagram 1}
 \end{center}   
\end{figure}
 
As a way of example, it follows from Figure \ref{Diagram 1} that there are at least two points in the intersection of curves in $C$ not in $E_i$ and $E_i$, which correspond to the two extreme edges of the graph.

\begin{lemma}\label{EC>1} For any $i$, we have $E_i\cdot C\geq 1$. 
\end{lemma}
\begin{proof} First we observe that if $C_j\subset E_i$, then $E_i\cdot C_j=-1$ is only possible for one $j$. Otherwise, we have $E_i\cdot C_j=0$. Since, there is a $(-1)$-curve $F\subset E_i$, and because of ampleness of $K_W$ then we have that $F\cdot C\geq 2$. Hence, we have that $E_i$ intersects with $C\setminus E_i$ in at least $2$. Thus, we conclude that $E_i\cdot C\geq 1$.
\end{proof}

\begin{remark}\label{FC>2} As we saw in the proof, we remark that for any $(-1)$-curve $F$ in $X$ we must have $F\cdot C\geq 2$. (This is because $K_W$ is ample.) Similarly, any $(-2)$-curve in $X$ must intersect the chain $C$ positively. In addition, note that we have $\sum_{i=1}^{m}E_i\cdot C \geq m+1$.
\end{remark}

The following example shows a sequence of accumulation points of $\{K^2\}$ on stable surfaces with only one cyclic singularity. It is constructed in a similar way to the one shown in \cite{blache1995example}. 

\begin{example}\label{Example $[4,n_0,4]$} 

Let $S' \to \P^1$ be an elliptic fibration obtained by blowing up at the intersection points of two general cubic curves in $\mathbb{P}^2$. It has $12$ nodal rational fibers (type $I_1$ according to Kodaira's notation). Now, let $n_0>0$ and let $f\colon S\to S'$ be the $n_0$-th cyclic cover (see e.g \cite{urzua2010arrangements}) branched along $F_1+\cdots +F_{n_0}$, where $F_i$ are general fibers on $S'$.

We have that $K_{S}^2=0$. Note that for every $(-1)$-curve $\beta$ in $S'$ the self-intersection of $f^*(\beta)$ is $-n_0$. Let us choose two nodal singular fibers $F$ and $F'$ in $S'$. After blowing up at the points on the nodes of $F$, and $F'$, we obtain a smooth projective surface $X_0$, which has the configuration $[4,n_0,4]$ shown in Figure \ref{Ex1}.
\begin{figure}[H]
    \begin{center}
    \includegraphics[width=4.5cm]{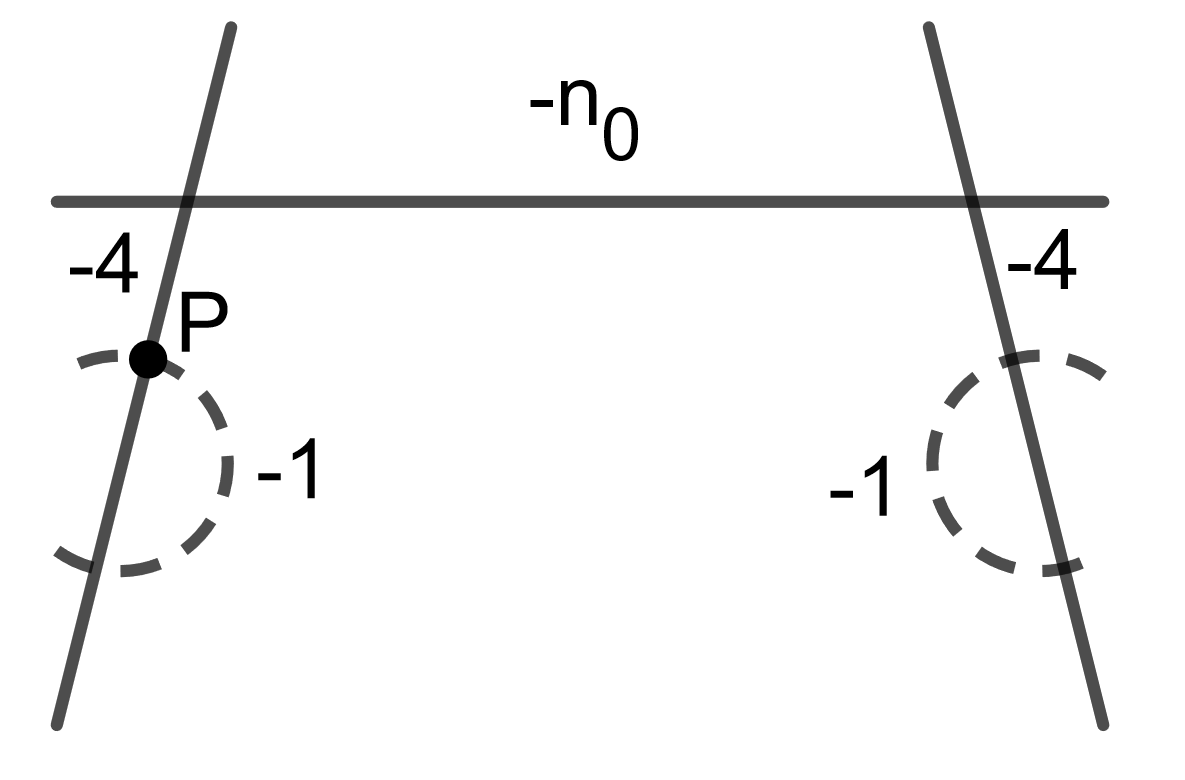}
    \caption{The configuration $[4,n_0,4]$.}
    \label{Ex1}
    \end{center}
    \end{figure}

Let us denote by $C_1,C_2$ and $C_3$ the curves in Figure \ref{Ex1} with self-intersection $-4,-n_0,$ and $-4$ respectively. We will construct a sequence of smooth projective surfaces $\{X_k\}$ by blowing up at points in $[4,n_0,4]$. First, let $X_1$ be the surface obtained by blowing up at the point $P$ in Figure \ref{Ex1}, and let $\Gamma$ the exceptional curve. The surface $X_2$ is obtained by blowing up at $C_1\cap \Gamma$, where $C_1$ is the strict transform. We continue blowing up at the point in the intersection between the strict transform of $C_1$ and the last exceptional curve obtained. After $k$ blow ups at points in the configuration, we obtain a smooth projective surface $X_k$, which has a chain of rational curves $C$. Each component $C_j$ of $C$ has a self-intersection in the sequence $\{-2,\dots,-2,-(4+k),-n_0,-4\}$, where $k$ is the number of $2$'s on the left side. By the construction, there is a $(-1)$-curve intersecting $C_1$, and $C_{k+1}$. By Artin's contractibility Theorem \cite[Thm. 2.3]{artin1962some}, we may contract $C$ to obtain a normal projective surface $W_k$ with only one cyclic quotient singularity. Note that there are $k+1$ divisors $E_j$ corresponding to the pull-back of the $(-1)$-curves of the blow downs. The graph of $\Gamma_{E_j}$ for $j=1,\dots, k$  is shown in Figure \ref{Ex3}, where $j$ is the number of $2$'s on the left side. 

\begin{figure}[H]
\begin{tikzpicture}[roundnode/.style={circle, draw=black, fill=white, thick,  scale=0.6},squarednode/.style={rectangle, draw=black, fill=white, thick, scale=0.7},roundnodefill/.style={circle, draw=black, fill=black, thick,  scale=0.6},roundnodewhite/.style={circle, draw=black, fill=white, thick,  scale=0.6},letra/.style={rectangle, draw=white, fill=white, thick, scale=0.7}]

\draw[black, thick] (-3,0) -- (-2,0){};
\draw[black, thick] (-2,0) -- (-1.4,0){};
\draw[black, thick] (-0.6,0) -- (0,0){};
\draw[black, thick] (0,0) -- (1,0){};
\draw[black, thick] (1,0) -- (2,0){};
\draw[black, thick] (2,0) -- (3,0){};

\node[squarednode] at (-3,0){} ;
\node[letra] at (-3.1,-0.5){$-2$} ;
\node[squarednode] at (-2,0){} ;
\node[letra] at (-2.1,-0.5){$-2$} ;
\node[squarednode] at (0,0){} ;
\node[letra] at (-0.1,-0.5){$-2$} ;

\node[roundnodefill] at (1,0){};
\node[letra] at (0.9,-0.5){$-(4+j)$} ;
\node[roundnodefill] at (2,0){};
\node[letra] at (1.9,-0.5){$-n_0$} ;
\node[roundnodefill] at (3,0){};
\node[letra] at (2.9,-0.5){$-4$} ;
\filldraw[black] (-1.2,0) circle (0.5pt) node[anchor=west] {};
\filldraw[black] (-1,0) circle (0.5pt) node[anchor=west] {};
\filldraw[black] (-0.8,0) circle (0.5pt) node[anchor=west] {};
\draw[black,  thick] (-3,0.1).. controls (-2,1) and (0,1).. (1,0.1);
\node[roundnodewhite] at (-0.9,0.76){};
\node[letra] at (-0.9,0.4) {$-1$};
\end{tikzpicture}
\caption{The graph of $\Gamma_{E_j}$.}
\label{Ex3}
\end{figure}
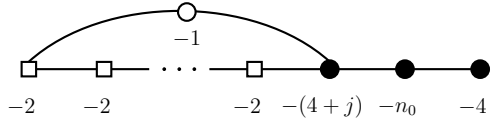

By pulling back the canonical divisor $K_{W_k}$, we can directly  write it as an effective sum of divisors, and then by the Nakai-Moishezon criterion, we obtain that $K_{W_k}$ is ample. By the canonical formula and induction over $k$, we obtain that 
$$\lim_{k\to\infty} K^2_{W_k}=\dfrac{4n_0^2-8n_0+2}{4n_0-1},$$ which is a sequence of accumulation points tending to $\infty$. 
\end{example}

\section{Bounding the case with one cyclic quotient singularity} 

In this section, we consider a normal stable surface $W$ with only one cyclic quotient singularity $P$, following the notation used previously. The goal is to show optimal bounds for the continued fraction associated to $P$. To start, we can easily see that if every exceptional divisor $E_i$ satisfies $E_i \cdot C \geq 2$, then we obtain the following bounds.

\begin{proposition}\label{delta=0}
Assume that $E_i \cdot C \geq 2$ for all $i$. Then 
\begin{equation}\label{E.C>2 sum}
    \sum_{j=1}^{r}\big( b_j-2\big)\leq 2(K_W^2-K_S^2)+2\bigg(\frac{2(n-1)-q-q'}{n}\bigg)-\pi^*K_S\cdot C ,
\end{equation}
and 
\begin{equation}\label{E.C>2 r}
r\leq 13K_W^2-2K_S^2+38 -\bigg(\frac{2+q+q'}{n}\bigg)-\pi^*K_S\cdot C.
\end{equation}

\end{proposition} 
\begin{proof} This corresponds to have $\delta=0$ in Theorem \ref{deltas}.
\end{proof}
In this way, if $K_S$ is nef, then we can bound singularities for all such $W$ with bounded $K_W^2$.



\begin{remark} In the particular case when $m=0$, we have that

\begin{equation*}
    \sum_{j=1}^{r}\big( b_j-2\big)= (K_W^2-K_S^2)+2-\bigg(\frac{2+q+q'}{n}\bigg),
\end{equation*}
and then, by using \eqref{temp5} we have
\begin{equation*}
r\leq 12K_W^2-K_S^2+36.
\end{equation*}
So, if $K_W\leq c$ for some positive number $c$, and $K_S$ is nef, we obtain finitely many options for $b_1,\dots,b_r$, and $r$. Thus, in what follows we will assume that $m>0$.
\end{remark}

Therefore, the critical case is when there exist an exceptional divisors $E_i$ such that $$E_i\cdot C=1.$$ 

\begin{remark}\label{tres casos} Using the same strategies as in the proof of Lemma \ref{EC>1}, we can see that if there are at least three points in the intersection of curves in $C\setminus E_i$ and curves in $E_i$ then we have that $E_i\cdot \big( \sum_{j=1}^{r}C_j\big)\geq 2$. Thus, if we have $E_i\cdot \big( \sum_{j=1}^{r}C_j\big)=1$ then there are two or fewer points on this intersection, and $\Gamma_{E_i}$ must be one of the following (see \cite[pp. 6]{rana2019optimal}): 


\begin{figure}[h!]
\begin{center}
\begin{tikzpicture}[roundnode/.style={circle, draw=black, fill=white, thick,  scale=0.6},squarednode/.style={rectangle, draw=black, fill=white, thick, scale=0.7},roundnodefill/.style={circle, draw=black, fill=black, thick,  scale=0.6}]

\draw[black, thick] (-4,0) -- (-3.4,0){};
\draw[black, thick] (-2.6,0) -- (-0.4,0){};
\draw[black, thick] (0.4,0) -- (2.6,0){};
\draw[black, thick] (3.4,0) -- (4,0){};

\node[roundnodefill] at (-4,0){} ;
\node[roundnodefill] at (-2,0){} ;
\node[squarednode] at (-1,0){};
\node[squarednode] at (1,0){};
\node[roundnodefill] at (2,0){} ;
\node[roundnodefill] at (4,0){} ;

\filldraw[black] (-3.2,0) circle (0.5pt) node[anchor=west] {};
\filldraw[black] (-3,0) circle (0.5pt) node[anchor=west] {};
\filldraw[black] (-2.8,0) circle (0.5pt) node[anchor=west] {};

\filldraw[black] (-0.2,0) circle (0.5pt) node[anchor=west] {};
\filldraw[black] (0,0) circle (0.5pt) node[anchor=west] {};
\filldraw[black] (0.2,0) circle (0.5pt) node[anchor=west] {};

\filldraw[black] (2.8,0) circle (0.5pt) node[anchor=west] {};
\filldraw[black] (3,0) circle (0.5pt) node[anchor=west] {};
\filldraw[black] (3.2,0) circle (0.5pt) node[anchor=west] {};

\end{tikzpicture}
\caption{Case A.}
  \label{Case 1.}
\end{center}
\end{figure}

\begin{figure}[H]
\begin{center}
\begin{tikzpicture}[roundnode/.style={circle, draw=black, fill=white, thick,  scale=0.6},squarednode/.style={rectangle, draw=black, fill=white, thick, scale=0.7},roundnodefill/.style={circle, draw=black, fill=black, thick,  scale=0.6}]

\draw[black, thick] (-4,0) -- (-3.4,0){};
\draw[black, thick] (-2.6,0) -- (-0.4,0){};
\draw[black, thick] (0.4,0) -- (1,0){};

\node[squarednode] at (-4,0){} ;
\node[squarednode] at (-2,0){} ;
\node[roundnodefill] at (-1,0){};
\node[roundnodefill] at (1,0){};

\filldraw[black] (-3.2,0) circle (0.5pt) node[anchor=west] {};
\filldraw[black] (-3,0) circle (0.5pt) node[anchor=west] {};
\filldraw[black] (-2.8,0) circle (0.5pt) node[anchor=west] {};

\filldraw[black] (-0.2,0) circle (0.5pt) node[anchor=west] {};
\filldraw[black] (0,0) circle (0.5pt) node[anchor=west] {};
\filldraw[black] (0.2,0) circle (0.5pt) node[anchor=west] {};
\end{tikzpicture}
\caption{Case B.}
  \label{Case 2.}
\end{center}
\end{figure}
\end{remark}

\begin{figure}[H]
\begin{center}
\begin{tikzpicture}[roundnode/.style={circle, draw=black, fill=white, thick,  scale=0.6},squarednode/.style={rectangle, draw=black, fill=white, thick, scale=0.7},roundnodefill/.style={circle, draw=black, fill=black, thick,  scale=0.6}]

\draw[black, thick] (-4,0) -- (-3.4,0){};
\draw[black, thick] (-2.6,0) -- (-0.4,0){};
\draw[black, thick] (0.4,0) -- (2.6,0){};
\draw[black, thick] (3.4,0) -- (4,0){};

\node[squarednode] at (-4,0){} ;
\node[squarednode] at (-2,0){} ;
\node[roundnodefill] at (-1,0){};
\node[roundnodefill] at (1,0){};
\node[squarednode] at (2,0){} ;
\node[squarednode] at (4,0){} ;

\filldraw[black] (-3.2,0) circle (0.5pt) node[anchor=west] {};
\filldraw[black] (-3,0) circle (0.5pt) node[anchor=west] {};
\filldraw[black] (-2.8,0) circle (0.5pt) node[anchor=west] {};

\filldraw[black] (-0.2,0) circle (0.5pt) node[anchor=west] {};
\filldraw[black] (0,0) circle (0.5pt) node[anchor=west] {};
\filldraw[black] (0.2,0) circle (0.5pt) node[anchor=west] {};

\filldraw[black] (2.8,0) circle (0.5pt) node[anchor=west] {};
\filldraw[black] (3,0) circle (0.5pt) node[anchor=west] {};
\filldraw[black] (3.2,0) circle (0.5pt) node[anchor=west] {};

\end{tikzpicture}
\caption{Case C.}
  \label{Case 3.}
\end{center}
\end{figure}
In order to describe in more detail the behavior of $\Gamma_{E_i}$, where $E_i$ is an exceptional divisor such that $E_i\cdot C=1$, we will introduce the following definition.
\begin{definition}\label{Long diagram} Let $k,l$ be positive integers. We say that $E_i$ has a \textit{long diagram} if $\Gamma_{E_i}$ is a diagram of type $(i)$, $(ii)$, $(iii)$, $(iv)$, and there is a $(-1)$-curve $F$ as shown in the following figures.


\begin{figure}[H]
\begin{center}
\begin{tikzpicture}[roundnode/.style={circle, draw=black, fill=white, thick,  scale=0.6},squarednode/.style={rectangle, draw=black, fill=white, thick, scale=0.7},roundnodefill/.style={circle, draw=black, fill=black, thick,  scale=0.6},roundnodewhite/.style={circle, draw=black, fill=white, thick,  scale=0.6},letra/.style={rectangle, draw=white, fill=white, thick, scale=0.7}]

\draw[black, thick] (-3,0) -- (-2.4,0){};
\draw[black, thick] (-1.6,0) -- (0.6,0){};
\draw[black, thick] (1.4,0) -- (2.6,0){};
\draw[black, thick] (3.4,0) -- (4,0){};

\draw[black,  thick] (-3,0.1).. controls (-2,1) and (1,1).. (2,0.1);
\node[roundnodewhite] at (-0.5,0.76){};
\node[letra] at (-0.5,0.4) {-1};


\node[squarednode] at (-3,0){} ;
\node[letra] at (-3.1,-0.5){$-2$} ;

\filldraw[black] (-1.8,0) circle (0.5pt) node[anchor=west] {};
\filldraw[black] (-2,0) circle (0.5pt) node[anchor=west] {};
\filldraw[black] (-2.2,0) circle (0.5pt) node[anchor=west] {};

\node[squarednode] at (-1,0){} ;
\node[letra] at (-1.1,-0.5){$-2$} ;

\node[roundnodefill] at (0,0){};
\node[letra] at (-0,-0.5){$C_{l+1}$} ;

\filldraw[black] (0.8,0) circle (0.5pt) node[anchor=west] {};
\filldraw[black] (1,0) circle (0.5pt) node[anchor=west] {};
\filldraw[black] (1.2,0) circle (0.5pt) node[anchor=west] {};

\node[roundnodefill] at (2,0){};
\node[letra] at (2,-0.5){$-b$} ;

\filldraw[black] (2.8,0) circle (0.5pt) node[anchor=west] {};
\filldraw[black] (3,0) circle (0.5pt) node[anchor=west] {};
\filldraw[black] (3.2,0) circle (0.5pt) node[anchor=west] {};

\node[roundnodefill] at (4,0){};
\node[letra] at (4.9,-0.5){} ;

\end{tikzpicture}
\caption{Diagram of type (i).}
  \label{Diagram (i)}
\end{center}
\end{figure}
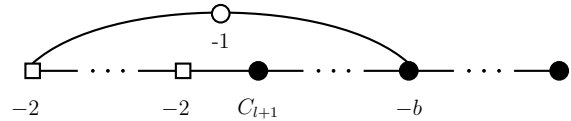

\begin{figure}[H]
\begin{center}
\begin{tikzpicture}[roundnode/.style={circle, draw=black, fill=white, thick,  scale=0.6},squarednode/.style={rectangle, draw=black, fill=white, thick, scale=0.7},roundnodefill/.style={circle, draw=black, fill=black, thick,  scale=0.6},roundnodewhite/.style={circle, draw=black, fill=white, thick,  scale=0.6},letra/.style={rectangle, draw=white, fill=white, thick, scale=0.7}]

\draw[black, thick] (-3,0) -- (-2.4,0){};
\draw[black, thick] (-1.6,0) -- (0.6,0){};
\draw[black, thick] (1.4,0) -- (2.6,0){};
\draw[black, thick] (3.4,0) -- (4,0){};

\draw[black,  thick] (-1,0.1).. controls (0,1) and (1,1).. (2,0.1);
\node[roundnodewhite] at (0.5,0.76){};
\node[letra] at (0.5,0.4) {-1};


\node[squarednode] at (-3,0){} ;
\node[letra] at (-3.1,-0.5){$-2$} ;

\filldraw[black] (-1.8,0) circle (0.5pt) node[anchor=west] {};
\filldraw[black] (-2,0) circle (0.5pt) node[anchor=west] {};
\filldraw[black] (-2.2,0) circle (0.5pt) node[anchor=west] {};

\node[squarednode] at (-1,0){} ;
\node[letra] at (-1.1,-0.5){$-2$} ;

\node[roundnodefill] at (0,0){};
\node[letra] at (-0.0,-0.5){$C_{l+1}$} ;

\filldraw[black] (0.8,0) circle (0.5pt) node[anchor=west] {};
\filldraw[black] (1,0) circle (0.5pt) node[anchor=west] {};
\filldraw[black] (1.2,0) circle (0.5pt) node[anchor=west] {};

\node[roundnodefill] at (2,0){};
\node[letra] at (2,-0.5){$-b$} ;

\filldraw[black] (2.8,0) circle (0.5pt) node[anchor=west] {};
\filldraw[black] (3,0) circle (0.5pt) node[anchor=west] {};
\filldraw[black] (3.2,0) circle (0.5pt) node[anchor=west] {};

\node[roundnodefill] at (4,0){};
\node[letra] at (4.9,-0.5){} ;

\end{tikzpicture}
\caption{Diagram of type (ii).}
  \label{Diagram (ii)}
\end{center}
\end{figure}
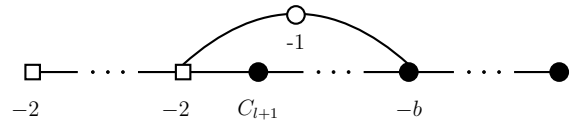


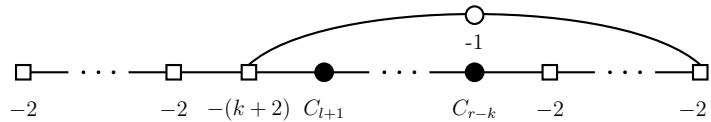
\begin{figure}[H]
\begin{center}

\begin{tikzpicture}[roundnode/.style={circle, draw=black, fill=white, thick,  scale=0.6},squarednode/.style={rectangle, draw=black, fill=white, thick, scale=0.7},roundnodefill/.style={circle, draw=black, fill=black, thick,  scale=0.6},roundnodewhite/.style={circle, draw=black, fill=white, thick,  scale=0.6},letra/.style={rectangle, draw=white, fill=white, thick, scale=0.7}]

\draw[black, thick] (-5,0) -- (-4.4,0){};
\draw[black, thick] (-3.6,0) -- (-0.4,0){};
\draw[black, thick] (0.4,0) -- (2.6,0){};
\draw[black, thick] (3.4,0) -- (4,0){};

\node[squarednode] at (-5,0){} ;
\node[letra] at (-5,-0.5){$-2$} ;

\filldraw[black] (-3.8,0) circle (0.5pt) node[anchor=west] {};
\filldraw[black] (-4,0) circle (0.5pt) node[anchor=west] {};
\filldraw[black] (-4.2,0) circle (0.5pt) node[anchor=west] {};

\node[squarednode] at (-3,0){} ;
\node[letra] at (-3,-0.5){$-2$} ;

\node[squarednode] at (-2,0){} ;
\node[letra] at (-2,-0.5){$-(k+2)$} ;

\node[roundnodefill] at (-1,0){};
\node[letra] at (-1.0,-0.5){$C_{l+1}$} ;

\filldraw[black] (-0.2,0) circle (0.5pt) node[anchor=west] {};
\filldraw[black] (0,0) circle (0.5pt) node[anchor=west] {};
\filldraw[black] (0.2,0) circle (0.5pt) node[anchor=west] {};

\node[roundnodefill] at (1,0){};
\node[letra] at (1,-0.5){$C_{r-k}$} ;

\node[squarednode] at (2,0){} ;
\node[letra] at (2,-0.5){$-2$} ;


\filldraw[black] (2.8,0) circle (0.5pt) node[anchor=west] {};
\filldraw[black] (3,0) circle (0.5pt) node[anchor=west] {};
\filldraw[black] (3.2,0) circle (0.5pt) node[anchor=west] {};

\node[squarednode] at (4,0){} ;
\node[letra] at (3.9,-0.5){$-2$} ;

\draw[black,  thick] (-2,0.1).. controls (-1,1) and (3,1).. (4,0.1);
\node[roundnodewhite] at (1,0.76){};
\node[letra] at (1,0.4) {-1};

\end{tikzpicture}
\caption{Diagram of type (iii).}
  \label{Diagram (iii)}
\end{center}
\end{figure}

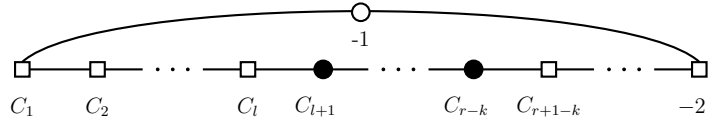
\begin{figure}[H]
\begin{center}
\begin{tikzpicture}[roundnode/.style={circle, draw=black, fill=white, thick,  scale=0.6},squarednode/.style={rectangle, draw=black, fill=white, thick, scale=0.7},roundnodefill/.style={circle, draw=black, fill=black, thick,  scale=0.6},roundnodewhite/.style={circle, draw=black, fill=white, thick,  scale=0.6},letra/.style={rectangle, draw=white, fill=white, thick, scale=0.7}]

\draw[black, thick] (-5,0) -- (-3.4,0){};
\draw[black, thick] (-2.6,0) -- (-0.4,0){};
\draw[black, thick] (0.4,0) -- (2.6,0){};
\draw[black, thick] (3.4,0) -- (4,0){};

\node[squarednode] at (-5,0){} ;
\node[letra] at (-5,-0.5){$C_1$} ;

\node[squarednode] at (-4,0){} ;
\node[letra] at (-4,-0.5){$C_2$} ;

\filldraw[black] (-2.8,0) circle (0.5pt) node[anchor=west] {};
\filldraw[black] (-3,0) circle (0.5pt) node[anchor=west] {};
\filldraw[black] (-3.2,0) circle (0.5pt) node[anchor=west] {};

\node[squarednode] at (-2,0){} ;
\node[letra] at (-2,-0.5){$C_l$} ;

\node[roundnodefill] at (-1,0){};
\node[letra] at (-1.1,-0.5){$C_{l+1}$} ;

\filldraw[black] (-0.2,0) circle (0.5pt) node[anchor=west] {};
\filldraw[black] (0,0) circle (0.5pt) node[anchor=west] {};
\filldraw[black] (0.2,0) circle (0.5pt) node[anchor=west] {};

\node[roundnodefill] at (1,0){};
\node[letra] at (0.9,-0.5){$C_{r-k}$} ;

\node[squarednode] at (2,0){} ;
\node[letra] at (2,-0.5){$C_{r+1-k}$} ;


\filldraw[black] (2.8,0) circle (0.5pt) node[anchor=west] {};
\filldraw[black] (3,0) circle (0.5pt) node[anchor=west] {};
\filldraw[black] (3.2,0) circle (0.5pt) node[anchor=west] {};

\node[squarednode] at (4,0){} ;
\node[letra] at (3.9,-0.5){$-2$} ;

\draw[black,  thick] (-5,0.1).. controls (-4,1) and (3,1).. (4,0.1);
\node[roundnodewhite] at (-0.5,0.76){};
\node[letra] at (-0.5,0.4) {-1};

\end{tikzpicture}
\caption{Diagram of type (iv).}
  \label{Diagram (iv)}
\end{center}
\end{figure}
\end{definition}

Now, we order the set of exceptional divisors according to their graph. Indeed, we say that $\Gamma_{E_i}$ is a subtree of $\Gamma_{E_j}$ if every $\square$-vertex of $\Gamma_{E_i}$ is a $\square$-vertex of $\Gamma_{E_j}$. Note that the set of the graphs $\Gamma_{E_i}$ is a partially ordered set with the following order: 
\begin{equation*}
\Gamma_{E_i}\leq \Gamma_{E_j}\Longleftrightarrow \Gamma_{E_i}  \textrm{ is a subtree of } \Gamma_{E_j}.
\end{equation*}

\begin{definition}\label{maximal graph} We say that the graph $\Gamma_{E_i}$ is called \textit{maximal} if it is a maximal element with respect $\leq$, and $E_i\cdot C=1$.
\end{definition}


By adding the discarded cases of Lemma $2.7$ in \cite{rana2019optimal}, which is valid in the context of T-singularities, we obtain the following result in the general case of cyclic quotient singularities. 


\begin{lemma}\label{E.C=1} Suppose that $E_i\cdot C=1$ for some $i$. Then $E_i$ has a long diagram. Moreover, if $E_i$ has a diagram of type $(iv)$, and $\Gamma_{E_i}$ is maximal, then there exists a sequence $\{m_1,\dots,m_s\}$ of natural numbers such that $C$ has continued fraction

\begin{equation*}
    [\dots,2,2+m_3,\underbrace{2,\dots,2}_{m_2-1},2+m_1,a_1,\dots, a_t,\underbrace{2,\dots, 2}_{m_1-1},2+m_2,2,\dots],
\end{equation*}
where $-a_1,\dots,-a_t$ correspond to the self-intersection of the $\bullet$ curves in $\Gamma_{E_i}$, and one of the ends is $2$ and the other one is $2+m_s$.

\end{lemma}

\begin{proof}
We divide this proof into the three cases of Remark \ref{tres casos}. In the first two cases, the argument is the one used in Lemma $2.7$ in \cite{rana2019optimal} for T-singularities.

\textbf{Case (1).} Assume that $E_i$ has the diagram shown in Figure \ref{Case 1.}. Because of the ampleness of $K_W$ we obtain that there is a $(-1)$-curve $F$ in $E_i$ which intersects $C$ twice (see Remark \ref{FC>2}). In this situation, we would obtain either a loop in $E_i$ or a third point of intersection with $C\setminus E_i$. But these cannot happen because $E_i$ is a tree of rational curves and $E_i\cdot C=1$.  

For the next case, we will denote by $C_1,\ldots, C_l$ the $\square$ curves on the left side in the diagrams shown in Figure \ref{Case 2.} and \ref{Case 3.}.

\textbf{Case (2).} Suppose that $E_i$ has the diagram shown in Figure \ref{Case 2.}. By the same argument done in Case (1), there exists a $(-1)$-curve $F$ in $E_i$ which intersects a $\square$ curve $C_j$, and a $\bullet$ curve $C_{j'}$, in both cases transversally. Note that there are no more intersections of $F$ and curves in $E_i$, because otherwise we will have a loop in $E_i$.

In what follows, we will prove that the $(-1)$-curve $F$ must intersect $C$ as is shown in Figure \ref{Diagram (i)} or in Figure \ref{Diagram (ii)}, and that the $\square$ curves are $(-2)$-curves. Indeed, we first claim that $C_j^2=-2$. Otherwise, we would need other $(-1)$-curve disjoint to $F$ to continue contracting $E_i$, but this situation gives from the beginning either a cycle in $E_i$ or a third point of intersection with $C\setminus E_i$. So, we have $C_j^2=-2$.

Now, we note that if $C_j$ had two $\square$ neighbors, then $F$ would have multiplicity at least 2 in $E_i$, which violates that $E_i\cdot C=1$. So we have that $C_j=C_1$ or $C_j=C_l$ (see Figure \ref{Diagram (i)} and Figure \ref{Diagram (ii)}).

On the other hand, note that after contracting $C_j$, we will have the same situation above for the curve $C_2$ or $C_{l-1}$ respectively. Thus, applying the same argument above, we obtain that all curves $C_1,\dots ,C_l$ are $(-2)$-curves. So, we conclude that $E_i$ either has a diagram of type $(i)$ or $(ii)$ (see Definition \ref{Long diagram}).

For the last case, we will denote by $C_1, \dots, C_l$ the $\square$   curves on the left side of $\Gamma_{E_i}$, and
$C_{r-k+1},\dots,C_r$ the $\square$  curves on the right side.

\textbf{Case (3).} Suppose that $E_i$ has the diagram shown in Figure \ref{Case 3.}. By Remark \ref{FC>2}, we have that a $(-1)$-curve $F$ in $E_i$ intersects $C$ twice. In this case, the curve $F$ must intersect one $\square$ curve $C_j$ on the left, and one $\square$ curve $C_{j'}$ on the right (see Figure \ref{Case 2-1.8}).
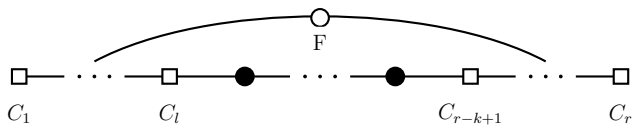
\begin{figure}[h!]
\begin{center}
\begin{tikzpicture}[roundnode/.style={circle, draw=black, fill=white, thick,  scale=0.6},squarednode/.style={rectangle, draw=black, fill=white, thick, scale=0.7},roundnodefill/.style={circle, draw=black, fill=black, thick,  scale=0.6},roundnodewhite/.style={circle, draw=black, fill=white, thick,  scale=0.6},letra/.style={rectangle, draw=white, fill=white, thick, scale=0.7}]

\draw[black, thick] (-3.4,0) -- (-4,0){};
\draw[black, thick] (-0.4,0) -- (-2.6,0){};

\draw[black, thick] (0.4,0) -- (2.6,0){};
\draw[black, thick] (3.4,0) -- (4,0){};

\node[squarednode] at (-4,0){} ;
\node[letra] at (-4,-0.5){$C_1$} ;
\node[squarednode] at (-2,0){} ;
\node[letra] at (-2,-0.5){$C_l$} ;

\node[roundnodefill] at (-1,0){};
\node[roundnodefill] at (1,0){};

\node[squarednode] at (2,0){} ;
\node[letra] at (2,-0.5){$C_{r-k+1}$} ;
\node[squarednode] at (4,0){} ;
\node[letra] at (4,-0.5){$C_{r}$} ;

\filldraw[black] (-3.2,0) circle (0.5pt) node[anchor=west] {};
\filldraw[black] (-3,0) circle (0.5pt) node[anchor=west] {};
\filldraw[black] (-2.8,0) circle (0.5pt) node[anchor=west] {};

\filldraw[black] (-0.2,0) circle (0.5pt) node[anchor=west] {};
\filldraw[black] (0,0) circle (0.5pt) node[anchor=west] {};
\filldraw[black] (0.2,0) circle (0.5pt) node[anchor=west] {};

\filldraw[black] (2.8,0) circle (0.5pt) node[anchor=west] {};
\filldraw[black] (3,0) circle (0.5pt) node[anchor=west] {};
\filldraw[black] (3.2,0) circle (0.5pt) node[anchor=west] {};

\draw[black,  thick] (-3,0.2).. controls (-1.5,1) and (1.5,1).. (3,0.2);
\node[roundnodewhite] at (0,0.78){};
\node[letra] at (0,0.45) {F};
\end{tikzpicture}
\caption{Case (3), and $(-1)$-curve $F$.}
 \label{Case 2-1.8}
  
\end{center}
\end{figure}

We first claim that $C_j^2=-2$ or $C_{j'}^2=-2$. On the contrary, we would need another $(-1)$-curve to contract them, but this would give either a loop in $E_i$ or a third point of intersection with $C$. 

Let us say $C_{j'}^2=-2$, then we must have that $C_{j'}=C_r$. Indeed, if we suppose that $C_{j'}$ has two $\square$ neighbors, then after contracting $F$ and $C_{j'}$, we will have a triple point in some $E_j$. But, it is not possible because $E_j$ is a normal simple crossings tree of rational curves. Thus, we have that $C_j'$ is one of the curves $C_{r-k+1}$ or $C_r$. Now, if we had $C_{j'}=C_{r-k+1}$ then $C_{r-k+1}$ would have multiplicity at least 2 in $E_i$, which contradicts the fact that $E_i\cdot C=1$. Thus, we obtain that $C_{j'}=C_r$.

On the other hand, we will prove that $C_j$ must be either $C_1$ or $C_l$. Otherwise, assume that $C_j$ has two $\square$ neighbors. Let $i'$ be a index such that $C_{r-i'+1}^2=\cdots =C_{r}^2=-2$, and $C_{r-i'}<-2$. We know that $C_{r-k}$ is not a curve in $E_i$, then we have $C_{r-k}^2<-2$, and $i'\leq k$. Assume that after blowing down $F,C_{r},\dots , C_{r-i'+1}$, we have that $C_j$ becomes a $(-1)$-curve. If $i'=k$, then $C_{r-i'+1}$ would has multiplicity at least two in $E_i$, but then $E_i\cdot C>1$. If instead $i'<k$, then contracting those curves and $C_j$ would give a triple point, which is not possible. Thus, we have that $C_j$ does not become a $(-1)$-curve. In this situation, we must need another $(-1)$-curve $F'$ to contract $C_j$. If $F'$ is disjoint of $F,C_{r},\dots , C_{r-i'+1}$, then $F'$ must intersect a $\bullet$ curve. But this would impliy that $E_i\cdot C>1$. So, $F'$ must intersect some of the $C_{r},\dots , C_{r-i'+1}$ in $E_i$, which is not possible because $E_i$ does not have loops. Thus, the unique possible case is that $i'=k$. But this implies that $C_{r-k+1}$ would have multiplicity at least two in $E_i$, which violates that $E_i\cdot C=1$. Thus, we obtain that $C_j$ cannot have two $\square$ neighbors. So, we know that $C_j=C_1$ or $C_j=C_l$. In this situation, we have that $E_i$ has a diagram of type $(iv)$ if $C_j=C_1$. (See Definition \ref{Long diagram}). 

Assume that $C_j=C_l$. We want to show that $E_i$ has a diagram of type $(iii)$. Indeed, let $i'$ be the maximal number such that $C_{r}^2=\cdots =C_{r-i'+1}^2=-2$, and $C_{r-i'}^2<-2$. As we did before, we know that $i'\leq k$ because $C_{r-k}^2<-2$. Let us first suppose that $i'<k$. Note that if after blowing down $F$ and those $(-2)$-curves the curve $C_l$ becomes a $(-1)$-curve, then $C_l$ would have multiplicity at least $2$ in $E_i$ because $i'<k$. So, $C_l$ does not became a $(-1)$-curve. Then, we must need another $(-1)$-curve $F'$ to contract $C_l$. If $F'$ is a $(-1)$-curve at the beginning, this would imply either a loop in $E_i$ or a third point of intersection of $E_i$ with $C$, none of which is possible. So, we have that $F'$ must intersect some curve in $C_r,\ldots,C_{r-i'+1}$. It implies that $C_{r-k+1}$ would has multiplicity at least two in $E_i$ which violates the fact that $E_i\cdot C=1$. 

Thus we have that $i'=k$, that is $C_{r-k+1}^2=\cdots =C_{r}^2=-2$. In this case, after blowing down $F$ and those $(-2)$-curves, the curve $C_{l}$ must become a $(-1)$-curve. On the contrary, we need another $(-1)$-curve intersecting $C_l$. That curve must be a $(-2)$- curve in the beginning in the process of contracting $E_i$ and it intersects the curve $C_{r+1-k}$, which implies that $C_{r+1-k}$ has at least multiplicity two in $E_i$. But this contradicts that $E_i\cdot C=1$. So, we have $C_l^2=-(k+2)$. By using a similar argument, it is shown that $C_1^2=\cdots C_{l-1}^2=-2$. Therefore, we have shown that $E_i$ has a diagram of type $(iii)$. 

For the last part of the proof, let us assume that $E_i$ has a diagram of type $(iv)$, and that $\Gamma_{E_i}$ is maximal. Let $m_s$ be the maximal number such that $C_{r-m_s}^2<-2$, and $C_r^2=\dots=C_{r-m_s+1}^2=-2$. If after blowing down the curves $F,C_r,\dots,C_{r-m_s+1}$ the curve $C_1$ does not became a $(-1)$-curve then we must need another $(-1)$-curve $F'$ in $E_i$ to contract $C_1$. If we have that $F'$ is a $(-1)$-curve at the beginning, it would imply either a loop in $E_i$ or a third point of intersection of $E_i$ with $C$. So, $F'$ must intersect the curves $F,C_r,\dots,C_{r-m_s+1}$, and then $C_{r-k+1}$ would has multiplicity at least two in $E_i$ which violates the fact that $E_i\cdot C=1$.

Thus, we have shown that $C_1$ is contracted after blowing down the curves $F,C_r,\dots,C_{r-m_s+1}$, and then $C_1^{2}=-(m_s+2)$, where $0< m_s \leq k-1$. We also note that after contracting the curves $F,C_r,\dots,C_{r-m_s+1}$ we obtain the same situation for the remaining curves in $C$, and then we can apply the same analysis. Therefore, we obtain that $C$ has continued fraction:
\begin{equation*}
    [\dots,2,2+m_3,\underbrace{2,\dots,2}_{m_2-1},2+m_1,a_1,\dots, a_t,\underbrace{2,\dots, 2}_{m_1-1},2+m_2,2,\dots],
\end{equation*}
where $-a_1,\dots,-a_t$ correspond to the self-intersection of the $\bullet$ curves in $\Gamma_{E_i}$, and $\{m_1,\dots,m_s\}$ is a fixed sequence of natural numbers. 
\end{proof}

\begin{remark} We recall that diagrams of type $(iv)$ were discarded on Lemma $2.7$ in \cite{rana2019optimal} for T-singularities, because of the ampleness of $K_W^2$, we cannot have a $(-1)$-curve intersecting both ends in a T-configuration. Also, diagrams of type $(iii)$ were discarded on Lemma $2.7$ in \cite{rana2019optimal} for T-singularities, because we cannot have a T-configuration with $(-2)$-curves in both ends.
\end{remark}

\begin{remark}\label{sequence E} Let $E_i$ be an exceptional divisor with diagram of type $(iv)$ such that $\Gamma_{E_i}$ is maximal. Assume that $E$ is a pullback of a curve in $E_i$. We associate to $E$ the sub sequence of $\{m_1,\dots,m_s\}$ which corresponds to the curves $C_{m_j}$ with $C_{m_j}^2=-(2+m_j)$ that are contracted in $E$.
\end{remark}

\begin{remark}\label{Divisores maximales} We remark that only one of the following situations can happen.
\begin{itemize}

    \item We have that $E_i \cdot C \geq 2$ for all $i$.
    
    \item There is a unique exceptional divisor $E_i$ such that its graph is maximal.
    
    \item There are two exceptional divisors $E_i$, $E_j$ such that their graphs are maximal. Moreover, we have that $\Gamma_{E_i}$, $\Gamma_{E_j}$ must be of type $(i)$ or $(ii)$.
\end{itemize}

\end{remark}

\begin{notation}\label{def-delta} The number of exceptional divisors $E_j$ such that $E_j\cdot C=1$ will be denoted by $\delta$.
\end{notation}

\begin{lemma}\label{delta-bound} Under conditions of Theorem \ref{deltas}. Let $E_1, \dots ,E_m$ be the exceptional divisors defined after Diagram \eqref{Diagram}. (They satisfy $E_i^2=-1$ and $E_i \cdot E_j=0$.) We have that one of the following cases holds:
\begin{itemize}

\item[(A)] For every $i$ we have $E_i\cdot C\geq 2$. In this case $\delta=0$.

\item[(B)] There is a unique $\Gamma_{E_i}$ maximal graph. Assume that $E_i$ contains only the curves $C_1,\dots ,C_l$ in $C$. Then
\begin{itemize}

\item[(B.1)] $\delta=l$ if $E_i$ has a diagram of type $(i)$.

\item[(B.2)] $\delta=1$ if $E_i$ has a diagram of type $(ii)$.
\end{itemize}

\item[(C)] There is a unique $\Gamma_{E_i}$ maximal graph. Assume that $E_i$ contains only the curves $C_1,\dots ,C_l,C_{r+1-k},\dots ,C_r$ in $C$. Then    
\begin{itemize}

\item[(C.1)] $\delta= k+1$ if $E_i$ has a diagram of type $(iii)$.

\item[(C.2)] $\delta= k+l$ if $E_i$ has a diagram of type $(iv)$.
\end{itemize}

\item[(D)] There exist two maximal graphs $\Gamma_{E_i},\Gamma_{E_{i'}}$. Assume that $E_i$ only contains the curves $C_1,\dots ,C_l$, and that $E_{i'}$ only contains the curves $C_{r+1-k},\dots, C_r$ in $C$. Then 
\begin{itemize}
    \item[(D.1)] $\delta=l+k$ if $E_i$ and $E_{i'}$ have diagrams of type $(i)$.
    
   \item[(D.2)] $\delta=l+1$ if $E_i$ has a diagram of type $(i)$, and $E_{i'}$ has a diagram of type $(ii)$.
   
    \item[(D.3)] $\delta=2$ if $E_i$ and $E_{i'}$ have diagrams of type $(ii)$.
    \end{itemize}
\end{itemize}

\end{lemma}

\begin{proof} We divide the proof into the cases of the statement. If there are not exceptional divisor with a long diagram, then by Lemma \ref{EC>1}, and Lemma \ref{E.C=1} we have that $E_i\cdot C\geq 2$ for every $i$, and so $\delta=0$. This shows Case \textbf{(A)}. In what follows, we will suppose that there are exceptional divisors with a long diagram.

\textbf{(B)} Suppose that there is a unique $\Gamma_{E_i}$ which is maximal. Assume that $E_i$ contains only the curves $C_1,\dots ,C_l$ in $C$. So, by this assumption and Lemma \ref{E.C=1} we obtain that the graph $\Gamma_{E_i}$ is of type (i) or (ii), $C_j^2=-2$ for every $1\leq j \leq l$, and $C_{l+1}^2\leq -3$. Without loss of generality, assume that $\pi$ starts by blowing down $F$, where $F$ is the $(-1)$-curve in $E_i$, that is $E_m=F$. 

Let $E$ be an exceptional divisor such that $E\cdot C=1$. By Lemma \ref{E.C=1}, we have that $E$ has a long diagram. Since $\Gamma_{E_i}$ is the unique maximal graph, then $E$ has a diagram of type $(i)$ or $(ii)$, and it must have as components some of the $(-2)$-curves $\{C_1,\dots, C_l\}$ or maybe all of them; otherwise we would obtain another maximal graph. Note that if the $(-1)$-curve in the diagram of $E$ is not $F$, then we have either a loop in $E$ or $E\cdot C \geq 2$, thus $F\subseteq E$, and hence $E$ has a diagram of the same type as $E_i$.

Let us write $E=c_1F+c_1C_1+c_2C_2+\cdots +c_lC_l+D$, where $c_1\geq 1$, $c_i\geq 0$ for $i>1$, and $D$ is an effective divisor which has no components of $C$ in its support. By using $1=E\cdot C=c_1+D\cdot C$, we obtain that $c_1=1$, and $D\cdot C=0$. But if $D>0$, we have that to contract $D$, it must exist another curve $(-1)$ disjoint from C, which contradicts the condition $K_W$ ample, and then $D=0$. Thus, 

$$E=c_1F+c_1C_1+c_2C_2+\cdots +c_lC_l$$
where $c_1\geq 1$, $c_i\geq 0$ for $i>1$.

At the same time, in the process of contracting $F, C_1, \ldots, C_{l}$, we obtain the following exceptional divisors:

\begin{equation*}
E_{m-j}=\left\{\begin{array}{lll}
     F+C_1+\cdots +C_{j}&if& \Gamma_{E_i}\ is\ of\ type\ (i), 1\leq j\leq l.\\
     F+C_l+\cdots +C_{l+1-j}&if& \Gamma_{E_i}\ is\ of\ type\ (ii), 1\leq j\leq l.
         
\end{array}
\right.
\end{equation*}    

With this notation we have that $E_i=E_{m-l}$. If $\Gamma_{E_i}$ is of type $(i)$, analyzing the graph of $E_{m-j}$, we obtain that only for $0< j \leq l$ we could have $E_{m-j}\cdot C=1$.
In the case that $\Gamma_{E_i}$ is of type $(ii)$, we have that $E_i$ is the only divisor such that $E_{m-j}\cdot C=1$ in the list $0\leq j \leq l$. For case (i), we have $E_{m-j}\cdot C=1$ for all $j$.

Therefore, if $E_i$ has a diagram of type $(i)$ then $\delta=l$, and if $E_i$ has a diagram of type $(ii)$ then $\delta= 1$. (See \cite[Lemma 2.10]{rana2019optimal}).

\textbf{(C)} Suppose that there is a unique $\Gamma_{E_i}$ which is maximal. Assume that $E_i$ contains only the curves $C_1,\dots ,C_l,C_{r+1-k},\dots ,C_r$ in $C$. By Lemma \ref{E.C=1}, we have that $E_i$ can only have diagram of type $(iii)$ or $(iv)$. We assume that $\pi$ starts by blowing down the $(-1)$-curve $F$ in $E_i$, that is $E_m=F$. We recall that $C_1^2=-2$ or $C_r^2=-2$ (see proof of Lemma \ref{E.C=1}). Let us say that $C_r^2=-2$.

\textbf{(C.1)} Say that $E_i$ has a diagram of type $(iii)$. Let $E$ be a exceptional divisor such that $E\cdot C=1$, by Lemma \ref{E.C=1} we have that $E$ has a long diagram. So, because $\Gamma_{E_i}$ is maximal, then $E$ must have some of $C_1,\dots ,C_l,C_{r-k+1},\dots ,C_r$ (or maybe all of them) as components; otherwise we would obtain another maximal graph. We also note that the $(-1)$-curve in the diagram of $E$ is $F$. Otherwise, we would have either a loop in $E_i$ or $E_i\cdot C >1$, but this is not possible. So, we can write $E$ as follows.

\begin{equation*}
E=c_1C_1+\cdots +c_lC_l+c_{r-k+1}C_{r-k+1}+\cdots+ c_rC_r+(c_l+c_r)F+D,    
\end{equation*}
where $c_j\geq 0$, $c_r>0$, and $D$ is an effective divisor which has no components of $C$ in its support. Since, we have that $C_j^2=-2$ for $j\in \{1,\dots,l-1,r-k+1,\dots,r\}$, and $C_l^2=-(k+2)$ (see Figure \ref{Diagram (iii)}). Then,

\begin{equation}\label{temp2}
    1=E\cdot C=c_r-(c_1+(k-2)c_l)+D\cdot C. 
\end{equation}

Now, we prove that $D\cdot C=0$. On the contrary, suppose that $D\cdot C>0$. We first note that if $c_1=0$ then $c_2=\cdots=c_l=0$, since otherwise $E$ would not have a long diagram, and so $E\cdot C>1$. In this case, because $D$ is effective then by (\ref{temp2}) we obtain that $D\cdot C=0$. If instead $c_1>0$ then $c_2,\dots,c_l>0$ because $E$ has a long diagram. Observe that $D$ can only intersect one component $C_j$ of $E$, otherwise after contracting $D$, we would obtain a loop in $E$  which is not possible. Also, we note that $D$ does not intersect the curves $C_{l}$ or $C_{r-k+1}$, since otherwise we would have $c_l$ or $c_{r-k+1}>1$, and then $E\cdot C>1$. In addition, if $D$ intersects a curve $C_j$ for some $1\leq j \leq l-1$, then $c_l>1$ which violates the fact that  $E\cdot C=1$. Thus, the divisor $D$ could only intersect a component $C_j$ of $E$ for some $j=r-k+2,\dots,r$. Then, contracting $D$ does not affect the curves $C_1,\dots, C_l$, and so we obtain $c_1=\cdots=c_l=1$. Finally, because $c_{r-k+1}=1$ we obtain that $c_r\geq k$. Then, by (\ref{temp2}), we obtain $D\cdot C=k-c_r\leq 0$. In both cases, we conclude that $D\cdot C=0$.

Using the fact that $D\cdot C=0$, we will prove that $D=0$. Indeed, if we had that $D>0$, then to contract $D$ there must exist another $(-1)$-curve  disjoint from C (that is, a $(-1)$-curve from the beginning in $E$) because $D$ does not intersect $C$. But, this contradicts the condition $K_W$ ample, and then $D=0$. 

Therefore, the divisor $E$ shows up in the process of contracting $E_i$. Note that in that process, we obtain the exceptional divisors $E_{m-(j+1)}= F+C_r+\cdots +C_{r-j}$ where $0\leq j\leq k-1$, and that $E_{m-(k+j+1)}=(k+1)F+kC_r+\cdots +C_{r-k+1}+C_l+\cdots +C_{l-j}$ where $0\leq j\leq l-1$. Here, because $\Gamma_{E_i}$ is the unique maximal graph, then $E_i=E_{m-(k+l)}$. By analyzing the graph of $E_{m-j}$ and by using \eqref{temp2}, we obtain that $E\cdot C=1$ only for $E_i$ and $E_{m-(j+1)}$, where $0\leq j\leq k-1$. Thus, we have $\delta=k+1$.

\textbf{(C.2)} Say that $E_i$ has a diagram of type $(iv)$. Let $E$ be an exceptional divisor such that $E\cdot C=1$. As in the Case (C.1), we have that $E$ must have some of $C_1,\dots,C_l,C_{r-k+1},\dots,C_r$ (or maybe all of them) as components, and that the $(-1)$-curve in the diagram of $E$ is $F$. Thus, we can write $E$ as follows.

\begin{equation*}
E=c_1C_1+\cdots +c_lC_l+c_{r-k+1}C_{r-k+1}+\cdots+ c_rC_r+(c_1+c_r)F+D,    
\end{equation*}
where $c_j\geq 0$, $c_r>0$, and $D$ is an effective divisor which has no components of $C$ in its support. Now, we prove that $D\cdot C=0$. On the contrary, suppose that $D\cdot C>0$. Note that $D$ can only intersect one component $C_j$ of $E$, otherwise after contracting $D$, we would obtain a loop in $E$  which is not possible. In addition, we have that $D$ must be contracted after blowing down the curves $C_j$ in $E$. Otherwise, we would obtain a loop in $E$. Let $j\leq l$ the maximal number such that $c_j>0$, and let $j'\leq r$ the minimal number such that $c_{j'}>0$. If we have that $D$ intersect one component of $C$ in $E$, then we would have that $c_j>0$ or $c_{j'}>0$, neither of which is possible. (both imply $E\cdot C>1$). 

Therefore, we have that $D$ can only intersect one curve in $C$ which is not in $E$. Then to contract $D$ it must exist another $(-1)$-curve disjoint from the curves of $C$ in $E$ or a $(-2)$-curve intersecting $C_j$ or $C_{j'}$. But this is not possible because $K_W$ is ample and $E\cdot C=1$. Thus, $D\cdot C=0$.

As we proved in case (C.1), the fact $D\cdot C=0$ implies $D=0$. Thus, $E$ is a pullback of a curve in $E_i$. So, $E$ has a diagram of type $(i)$, and a $(-1)$-curve intersecting both ends of the chain or it has a diagram of type $(iv)$.

Now, we will show that $\delta=k+l$. Indeed, by following the notation described in Lemma \ref{E.C=1}, let $\{m_1,\dots,m_s\}$ be the sequence associated to $E_i$. Let us denote by $C_{m_j}$ to the curve with self-intersection $-(2+m_j)$ in $C$, and by $c_{m_j}$ its multiplicity in $E_i$. In addition, let us write $c_{m_{s+1}}:=c_r$, and $C_{m_{s+1}}:=C_r$ . With this notation, we have $c_1=c_{m_s}$, and then

\begin{equation}\label{temp3}
    E_i\cdot C=c_{m_s}(1-m_{s})+c_{m_{s+1}}+\sum_{j=1}^{s-1}c_{m_j}(-m_j). 
\end{equation}

We first show that $\delta= k+l$ in the case when $s=1$ (see Figure \ref{The case when $s=1$}). Note that in this case we have that $l=1$, and $k=m_1$.
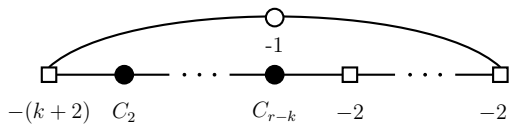
\begin{figure}[H]
\begin{center}

\begin{tikzpicture}[roundnode/.style={circle, draw=black, fill=white, thick,  scale=0.6},squarednode/.style={rectangle, draw=black, fill=white, thick, scale=0.7},roundnodefill/.style={circle, draw=black, fill=black, thick,  scale=0.6},roundnodewhite/.style={circle, draw=black, fill=white, thick,  scale=0.6},letra/.style={rectangle, draw=white, fill=white, thick, scale=0.7}]

\draw[black, thick] (-3,0) -- (-1.4,0){};
\draw[black, thick] (-0.6,0) -- (1.6,0){};
\draw[black, thick] (2.4,0) -- (3,0){};

\node[squarednode] at (-3,0){} ;
\node[letra] at (-3,-0.5){$-(k+2)$} ;

\node[roundnodefill] at (-2,0){};
\node[letra] at (-2.0,-0.5){$C_{2}$} ;

\filldraw[black] (-1.2,0) circle (0.5pt) node[anchor=west] {};
\filldraw[black] (-1,0) circle (0.5pt) node[anchor=west] {};
\filldraw[black] (-0.8,0) circle (0.5pt) node[anchor=west] {};

\node[roundnodefill] at (0,0){};
\node[letra] at (0,-0.5){$C_{r-k}$} ;

\node[squarednode] at (1,0){} ;
\node[letra] at (1,-0.5){$-2$} ;

\filldraw[black] (1.8,0) circle (0.5pt) node[anchor=west] {};
\filldraw[black] (2,0) circle (0.5pt) node[anchor=west] {};
\filldraw[black] (2.2,0) circle (0.5pt) node[anchor=west] {};

\node[squarednode] at (3,0){} ;
\node[letra] at (2.9,-0.5){$-2$} ;

\draw[black,  thick] (-3,0.1).. controls (-2,1) and (2,1).. (3,0.1);
\node[roundnodewhite] at (0,0.76){};
\node[letra] at (0,0.4) {-1};

\end{tikzpicture}
\caption{The case when $s=1$.}
\label{The case when $s=1$}
\end{center}
\end{figure}

Note that in this case, $E_i$ has also a diagram of type $(iii)$, and then $\delta=k+1$ as we proved in case (C.1). Here we obtain that $c_{m_1}=1$, $c_{m_2}=m_1$.


Assume $s>1$. We claim that $c_{m_j}=c_{m_{j-2}}+m_{j-1}c_{m_{j-1}}$ for $2<j\leq s+1$. Indeed, by Lemma \ref{E.C=1}, we have that $m_{j-1}-1$ is the number of the $(-2)$-curves between $C_{m_j}$ and $C_{m_{j-1}}$ in $C$. After contracting $C_{m_j}$ and such $(-2)$-curves, we obtain a SNC situation between $C_{m_{j-1}}$ and $C_{m_{j-2}}$. It follows from the form of contracting these curves that $c_{m_j}=c_{m_{j-2}}+m_{j-1}c_{m_{j-1}}$.

Moreover, if we assume that
\begin{equation}\label{temp4}
      c_{m_{s-2}}(1-m_{s-2})+c_{m_{s-1}}+\sum_{j=1}^{s-3}c_{m_j}(-m_j)=1, 
\end{equation}
then, by plugging $c_{m_j}=c_{m_{j-2}}+m_{j-1}c_{m_{j-1}}$ for $j=s,s+1$, and \eqref{temp4} in \eqref{temp3} , we obtain 

\begin{equation}\label{temp12}
E_i\cdot C=c_{m_s}(1-m_{s})+c_{m_{s+1}}+\sum_{j=1}^{s-1}c_{m_j}(-m_j)=1. 
\end{equation}

Note that for $E_i$ we have that $c_{m_1}=1$, and $c_{m_2}=m_1$. However, the analysis above does not depend on the initial values of $c_{m_j}$. It only depends on the form of contracting the curves in diagrams of type $(iv)$. For instance, one can have $c_{m_1}=\cdots=c_{m_j'}=0$ for some $j'\leq s$. Let $1\leq j'\leq s+1$ be the minimal number such that $C_{m_j'}$ is contracted in $E$. This implies that $c_{m_j'}>0$, and $c_{m_1}=\cdots=c_{m_{j'-1}}=0$. Since, we know that $E$ is a pullback of a curve in $E_i$, then by the form of contracting the curves in $E$  we obtain that $c_{m_j}>0$ for every $j\geq j'$. By using the form of the contraction in $E$, we compute that $c_{m_{j'+1}}=m_{j'}c_{m_{j'}}$  Thus, we know that $c_{m_j}=c_{m_{j-2}}+m_{j-1}c_{m_{j-1}}$ for every $j\geq j'+2$, and then we have that \eqref{temp3}, and \eqref{temp12} remain valid for $E$. So, we obtain that $E\cdot C=1$.

Therefore, because $k+l$ is the number of exceptional divisors $E$ that show up as pullback of curves in $E_i$, we obtain that $\delta=k+l$.


By Remark \ref{Divisores maximales}, we have the following last case.

\textbf{(D)} Suppose that there exist two maximal graphs $\Gamma_{E_i},\Gamma_{E_{i'}}$. Assume that $E_{i'}$ only contains the curves $C_{r+1-k},\dots, C_r$, and that $E_i$ only contains the curves $C_1,\dots ,C_l$ in $C$. In this case, by Lemma \ref{E.C=1} we have that $E_i$ and $E_i'$ can only have a diagram of type $(i)$ or $(ii)$, and we have that $C_1,\dots ,C_l,C_{r+1-k},\dots ,C_r$ are $(-2)$-curves. Because $E_i$, and $E_i'$ are maximal divisors such that their intersection with $C$ is equal to $1$, then we obtain that $C_{l+1}^2\leq -3$, and $C_{r-k}^2\leq -3$. As before, we assume that $\pi$ starts by blowing down the $(-1)$-curve $F$ in $E_i$. Let $F'$ be the $(-1)$-curve in the diagram of $E_i'$. We must have that $F\neq F'$, otherwise, we could not contract the divisor $E_{i'}$. Thus, we can describe separately the process to contract the divisors $E_i$, and $E_i'$ using Case (B). Combining the possible situations for $E_i$ and $E_{i'}$, we obtain the values described for $\delta$.

\end{proof}


\begin{remark}\label{no dependencia de la caracteristica algebraica} We recall the generalized Noether's inequality shown in \cite[Thm. 2.10]{tsunoda1992noether}, 

$$\chi(\mathcal{O}_W)\leq K_W^2+3.$$

We will use it in the proof of Theorem \ref{deltas} for bounding the length of a singularity only in terms of $K_W^2$.
\end{remark}

\begin{proof}[Proof of Theorem \ref{deltas}] We start by proving the initial inequality. It follows from the assumption that $\big(\sum_{i=1}^{m}E_i\big)\cdot C\geq 2m-\delta$. It follows from \eqref{canonical divisor formula WS}, and Lemma \ref{E.C} that

\begin{equation}\label{EQ4}
    \sum_{j=1}^{r}\big( b_j-2\big)\leq 2(K_W^2-K_S^2)+2\bigg(\frac{2(n-1)-q-q'}{n}\bigg)+\delta -\pi^*K_S\cdot C.
\end{equation}

In order to have the bound for the length we will use a generalization of the Bogomolov-Miyaoka-Yau inequality for orbifolds, log-BMY inequality for short (see e.g. \cite{langer2003logarithmic})

\begin{equation*}
    K_{W}^2\leq 3e_{orb}(W),
\end{equation*}
where $e_{orb}(W)$ is the orbifold Euler number of a quasiprojective surface $W$, with only isolated cyclic quotient singularities. That is defined as
\begin{equation*}
e_{orb}(W)=e(W)-\displaystyle{\sum\limits_{w\in \rm{Sing}(W)}\bigg(1-\frac{1}{|\pi_1(L_w)|} \bigg)},
\end{equation*}
where $L_w$ is the link of $w\in \rm{Sing}(W)$, and $\pi_1$ denotes the fundamental group. We recall that the germ of $w\in \rm{Sing}(W)$ is topologically the cone over a 3-manifold $S^3/G$, where $G\subset U(2,\mathbb{C})$ is a subgroup acting without fixed points. The 3-manifold $S^3/G$ is called the link of $w$.

In our case, the surface $W$ has $\rm{Sing}(W)=\{P\}$. Since the singularity $P$ is defined as the germ at the origin of a quotient of $\mathbb{C}^2$ by the group $\mathbb{Z}/n\mathbb{Z}$, we have that $|\pi_1(L_w)|=|\mathbb{Z}/n|=n$. Therefore,
\begin{equation*}
e_{orb}(W)=e(W)-\bigg(1-\frac{1}{n}\bigg).    
\end{equation*}

Plugging this formula together with $e(W)=e(X)-r$, and the Log-BMY inequality, we have
\begin{equation*}
    K_W^2\leq 3e(X)-3r-3\bigg(1-\frac{1}{n} \bigg).
\end{equation*}

By Proposition \ref{canonical divisor formula XW}, and Noether's formula for $X$, we have that
\begin{equation*}
    12\chi(\mathcal{O}_X)=K_W^2+A+e(X),
\end{equation*}
where $A=\sum_{j=1}^{r}\big(2-b_j\big)+\dfrac{2(n-1)-q-q'}{n}$. Putting these formulas together and using the fact that $\chi(\mathcal{O}_X)=\chi(\mathcal{O}_W)$ ($W$ has a rational singularity), we have that
\begin{equation}\label{EQ3.1}
    12\chi(\mathcal{O}_W)\leq 4e(X)-3r-3\bigg(1-\frac{1}{n} \bigg)+A.
\end{equation}

Also, by the Noether's formula for $X$, we have $4e(X)=48\chi(\mathcal{O}_X)-4K_X^2$. Replacing this in \eqref{EQ3.1}, and by Proposition \ref{canonical divisor formula XW}, we obtain that 
\begin{equation}\label{temp5}
    r\leq 12\chi(\mathcal{O}_W)-\dfrac{4}{3}K_W^2-A-\bigg(1-\frac{1}{n} \bigg).
\end{equation} 

Finally, using  \eqref{temp5} together with the generalized Noether's inequality (see Remark \ref{no dependencia de la caracteristica algebraica}), and the inequality for the sum $\sum_{j=1}^{r}\big(b_j-2\big)$ previously obtained, we have the boundedness for $r$.
\bigskip
\end{proof}




\begin{proof}[Proof of Corollary \ref{Cor1}] Assume that the number of $2$'s at the extremes of every $[b_1,\ldots,b_r] \in \sing(\mathcal{S})$ is bounded, say by a number $k>0$. Then, by Theorem \ref{deltas} we have that $\delta\leq k+1$, and because $K_S\geq 0$ for every surface in $\mathcal{S}$, we obtain that $r$, and $\sum_{j=1}^{r}b_j$ are bounded. So, we conclude that $Sing(\mathcal{S})$ is finite. On the other hand, if the number of $2$'s at the extremes of every $[b_1,\ldots,b_r] \in \sing(\mathcal{S})$ is not bounded, then we can construct infinitely many sets $\mathcal{S}$ of stable surfaces with $K_S$ nef, and $K_W^2\leq c$ such that $Sing(\mathcal{S})$ is infinite. (See e.g. Example \ref{Example $[4,n_0,4]$} with $n_0\geq c$).
\end{proof}


\begin{proof}[Proof of Corollary \ref{Cor2}] This follows from the facts that $K_S$ is nef, and that $\delta\leq 2$ in those cases.
\end{proof}


\section{Accumulation points for surfaces with one generalized T-singularity} 

This section describes the behavior of the accumulation points of volumes of stable surfaces with only one singularity belonging to a fixed family of generalized T-singularities. We also characterize the continued fractions that are admissible for chains. In the end, we show a way to construct accumulation points by starting in a stable surface with only one generalized T-singularity.  

\begin{proposition}\label{formation rule} Let $[b_1,\ldots,b_s]=n/q$ be a continued fraction. Let $q'$ be the inverse of $q$ modulo $n$ with $0<q'<n$, and let $m$ be the integer such that $qq'=1+mn$. Assume that $n>2$. Then, we have that $[b_1+1,b_2,\ldots,b_s,2]=N/Q$, where $N=2q-m+2n-q'$, and $Q=2q-m$. Moreover, we have that $[2,b_s,\ldots,b_2,b_1+1]=N/Q'$, where $Q'=q+n$.
\end{proposition}
\begin{proof} We know that $\frac{n}{q}=[b_1,\ldots,b_s]$ implies $[b_1,\ldots,b_s]=\frac{n}{q'}$, where $q'$ is the inverse of $q$ modulo $n$. So, we obtain that $[2,b_s,\ldots,b_1]=\frac{2n-q'}{n}$. Now, we would like to find $n'$ the inverse of $n$ modulo $2n-q'$. We observe that $n(2q-m)=q(2n-q')+1$, so $n'\equiv (2q-m) \ mod\ (2n-q')$. Thus, $[b_1,\ldots,b_s,2]=\frac{2n-q'}{n'}$.

We will prove that $n'=2q-m$. Observe that $m<q$. Otherwise, if we have that $m>q$, then $mn+1\geq qn>qq'$, but $qq'=mn+1$. Also, if we have $m=q$, then $q(q'-n)=1$. But this is impossible. So we know that $m<q$. In the same way, we obtain that $m<q'$. So, we have that $2q-m>0$. In addition, because $m+1\leq q'$, and $2n-q'> 2$ (if $n>2$), then we have that $2+m(2n-q')< q'(2n-q')$. So, we obtain that $(2q-m)< (2n-q')$ by using $q'>0$, and $qq'=1+mn$. Thus, $n'=2q-m$.

Therefore, we have that $[b_1+1,b_2,\ldots,b_s,2]=\frac{N}{Q}$, where $N=2q-m+2n-q'$ and $Q=2q-m$. 

Finally, we show that the inverse of $Q$ modulo $N$ is $q+n$. Indeed, note that $(2q-m)(q+n)\equiv 1\ mod (n)$. Also, by using that $m<q$ we obtain that $N=q+n+(q+n-m-q')>q+n>0$. Thus, we obtain that $[2,b_s,\ldots,b_2,b_1+1]=N/Q'$, where $Q'=q+n$.
\end{proof}

We will describe the Hirzebruch-Jung continued fractions, which are admissible for chains (see Definition \ref{admissible for chains}). Given a singularity $[b_1,\ldots,b_r]$, we say that a coefficient $b_i$ does not contract if the curve associated with $b_i$ does not. 

\begin{lemma}
If $[b_1,\ldots,b_r]$ is admissible for chains, then there are $b_i,b_j$ with $i\leq j$ such that  $$[b_1,\dots,b_r]-1-[b_1,\dots,b_r]-1-[b_1,\dots,b_r]$$ does not contract $b_i$ and $b_j$.
\label{center}
\end{lemma}

\begin{proof}
Otherwise we would have eventually inside of the contraction a situation $[1,1]$, and that makes the chain not admissible. In order to prove that, we will use induction on $r$. 


We first compute the base case for $r=1$. Here, to contract $b_1$ we must have either $b_1=2$ or $b_1=3$, and so we obtain either the situation $[1,1,1,2]$ or $[1,1]$, respectively.

Let us suppose that for every $[\textbf{a}]=[a_1,\ldots,a_k]$, and $k<r$ if $a_i$ are contracted for every $i$, then we obtain eventually the situation $[1,1]$ inside of the contraction $[\textbf{a}]-1-[\textbf{a}]-1-[\textbf{a}]$. Now, let $k=r$. We must have that $b_1=2$ or $b_r=2$. Assume, without loss of generality (we could flip the order), that $b_r=2$. Then, after contracting $b_r$, we have  
$$1+[b_1-1,\dots,b_{r-1},1,b_1-1,\dots,b_{r-1},1,b_1-1,\dots,b_{r-1},2],$$
and so, we obtain inside of it the situation $[\textbf{a}]-1-[\textbf{a}]-1-[\textbf{a}]$, where $[\textbf{a}]=[b_1-1,b_2,\ldots,b_{r-1}]$. By the inductive hypothesis, we conclude that the situation $[1,1]$ will appear inside of $[\textbf{a}]-1-[\textbf{a}]-1-[\textbf{a}]$, and so inside of $[b_1,\dots,b_r]-1-[b_1,\dots,b_r]-1-[b_1,\dots,b_r]$.
\end{proof}

The $[b_i,\ldots,b_j]$ is a sort of core which is necessary for the property admissible for chains. 

\begin{lemma}
Let $0<a<n$ be coprime integers, let $\frac{n}{a}=[x_1,\ldots,x_f]$ and $\frac{n}{n-a} =[y_1,\ldots,y_g]$. Then $$[x_1,\ldots,x_f,1,y_g,\ldots,y_1] =0.$$

\label{dual}
\end{lemma}

Lemma \ref{dual} is well-known, and it is the justification for the Riemenschneider's dot diagram. For example, if $\frac{n}{n-a}=[2,\ldots,2,y_i,\ldots,y_g$ where $y_i >2$, then $x_1=i-1+2=i+1$.



\begin{definition}\label{def core}
A \textit{core} is a Hirzebruch-Jung continued fraction $[e_1,\ldots,e_s]$ such that $e_i >1$ for all $i$ and either 

\begin{itemize}
\item[(1)] $s=1$ and $e_1 \geq 4$
\item[(2)] $s\neq 1$ and $e_1 \geq 3$ and $e_s \geq 3$.
\end{itemize}
\end{definition}

In this way, the limit cases $[4]$ and $[3,e_2,\ldots,e_{s-1},3]$ are cores, and with $e_i=2$ for $i=2,\ldots,s-1$, they are exactly the cores of T-chains. One can check by a direct computation that every core is admissible for chains. The remarkable fact is that all $[b_1,\ldots,b_r]$ admissible for chains are constructed from a core following the formation rule of T-chains. 

\begin{theorem}\label{admissibility}
Let $[b_1,\ldots,b_r]$ be an admissible for chains continued fraction. Then there is a unique core $[e_1,\ldots,e_s]$ such that $[b_1,\ldots,b_r]$ is obtained by applying the T-chain algorithm to $[e_1,\ldots,e_s]$.
\end{theorem}

\begin{proof} Consider the center $[b_i,\ldots,b_j]$ of $[b_1,\ldots,b_r]$ as shown in Lemma \ref{center}, adding the condition that $i$ is the minimal index such that $b_i$ is not contracted. Similarly, we ask for $j$ to be the maximal index such that $b_j$ is not contracted.

First, we assume that $i=1$, and $j=r$. In this case, we obtain directly that $[b_1,\ldots,b_r]$ is a core. So, we take $[e_1,\ldots,e_s]=[b_1,\ldots,b_r]$. 

In what follows, we will suppose that $1<i$ or $j<r$. Let us write 
$$[b_1,\ldots,b_r]=[a_1,\ldots,a_u,b_i,\ldots,b_j,c_1,\ldots,c_v],$$
of course keeping the position of $[b_i,\ldots,b_j]$. Note that the initial conditions over $i,j$ imply that $[c_1,\ldots,c_v,1,a_1,\ldots,a_u]$ will disappear. Assume, without loss of generality (we could flip the order), that $a_u$ is the last curve that disappears. In particular, we have that $1<i$. Now, we should treat $j=r$ and $j<r$ separately.  

\textbf{Case A.} Say $j=r$. Observe that $[a_1,\ldots,a_u]=[2,\ldots,2]$, and that $$[u+1,1,a_1,\ldots,a_u]=0.$$ As $b_i,b_j$ have to survive, we obtain that $b_i\geq 3$ and $b_j\geq u+3$. 

In this case, we have that $[b_1,\ldots,b_r]$ is obtained by applying the T-chain algorithm to $[e_1,\dots,e_s]:=[b_i,\dots,b_{j-1},b_j-u]$. We observe that if $i<j$ then $[e_1,\dots,e_s]$ is a core. However, if $i=j$ we have to prove that $b_j-u\geq 4$. On the contrary, let us suppose that $b_j-u=3$. Then, we obtain the situation $[a_1,\ldots,a_u,b_j-(u+1),b_j-(u+2),b_j-1]=0$ inside of

\begin{equation}\label{temp6}
    [b_1,\dots,b_r]-1-[b_1,\dots,b_r]-1-[b_1,\dots,b_r].
\end{equation}

But, this is impossible because $b_j=u+3$ must survive. So, we obtain that $b_j\geq 4$. Thus, we know that $[e_1,\dots,e_s]$ is a core.

\textbf{Case B.} Say $j<r$. Let $[a_1,\ldots,a_u]=[a_1,\ldots,a_w,2,\ldots,2]$ with $a_w>2$ and say that the number of $2$'s at the end is $l$. Then, one can check that $$[l+2,c_1,\ldots,c_v,1,a_1,\ldots,a_w,2,\ldots,2]=0.$$

As $b_i,b_j$ have to survive, we know that $b_i \geq 3$ and $b_j \geq l+4$. Let $[e_1,\dots,e_s]=[b_i,\dots,b_{j-1},b_j-(l+1)]$. The Riemenschneider's dot diagram will then give the algorithm from T-chains. 

On the other hand, we note that $[e_1,\dots,e_s]$ is a core if $i<j$. Similarly to Case A, if $i=j$ and $b_j=l+4$ then we obtain $[a_1,\ldots,a_u,b_j-(u+2),b_j-(u+3),b_j-1,c_1,\ldots,c_v]=0$ inside of \eqref{temp6}. But, this is impossible. So, if $i=j$ then $b_j\geq l+5$. Thus, we have that $[e_1,\dots,e_s]$ is a core. Due to the choice of $[b_i,\ldots,b_j]$, we conclude that $[e_1,\ldots,e_s]$ is unique.
\end{proof}


\begin{remark} In particular, by Theorem \ref{admissibility} we have that a generalized T-singularity $[b_1,\ldots,b_r]$ fulfills the condition $b_1>2$ or $b_r>2$.
\end{remark}


\begin{corollary}\label{descartar tipo iii} Let $W$ be a stable surface with a unique generalized T-singularity of center $[b_1,\dots,b_{r}]$. Then, we have either $E_i\cdot C\geq 2$ for every exceptional divisor $E_i$ or there is a unique exceptional divisor $E_i$ such that $\Gamma_{E_i}$ is maximal. In this case, we have that $\Gamma_{E_i}$ cannot be a diagram of type $(iii)$.
\end{corollary}


\begin{proof} By Theorem \ref{admissibility} we have that $b_1>2$ or $b_r>2$. So, we cannot have two maximal graphs nor a graph of type $(iii)$. Then, by Remark \ref{Divisores maximales}, and by Lemma \ref{E.C=1} we conclude that either $E_i\cdot C\geq 2$ for every exceptional divisor or there is a unique exceptional divisor $E_i$ with a maximal graph such that $\Gamma_{E_i}$ is of type $(i),(ii)$ or $(iv)$. 
\end{proof}


\begin{definition}\label{minimal center} Let $[e_1,\ldots,e_s]$ be a core, we say that $[e_1,\ldots,e_s]$ is \textit{minimal} if it cannot be obtained from another core $[b_1,\dots,b_r]$ by inserting $1$'s (see Definition \ref{generalized T-singularity}, (i)).
\end{definition}

\begin{remark} Let $[b_1,\ldots,b_r]$ be an admissible for chains continued fraction, and let $[e_1,\ldots,e_s]$ be its associated core (see Theorem \ref{admissibility}). It is immediate that the set of generalized T-singularities of center $[b_1,\ldots,b_r]$, is contained in the set of generalized T-singularities of center $[e_1,\dots,e_s]$. In particular, this is true if the core is minimal. For instance, the family of T-singularities is obtained by starting with the minimal core $[4]$. We classify the minimal cores in Proposition \ref{description minimal centers}.
\end{remark}
\begin{proposition}\label{description minimal centers} A core $[e_1,\dots,e_s]$ is minimal if and only if one of the following cases holds:

\begin{itemize}
    \item[(i)] $s$ is a prime number and $[e_1,\ldots,e_s]\neq [e_1,e_1-1,\ldots,e_1-1,e_1]$.
    
    \item[(ii)] $s$ is not prime and for every $1<u<s$ divisor of $s$ (say $s=ur$), either   
    \begin{itemize}
        \item there exist $2\leq i<r$, and $1\leq j<u$ such that $e_i\neq e_{i+jr}$, 
        
        \item there exists $1\leq j<u$ such that $e_{1+jr}+1\neq e_1 $, or
        
        \item there exists $1\leq j<u$ such that $e_{r+jr}+1\neq e_s$.
    \end{itemize}
\end{itemize}
\end{proposition}

\begin{proof} Let $[b_1,\ldots,b_r]$ be a core, and let $[c_1,\ldots,c_{kr}]$ be the continued fraction $[b_1,\dots,b_r,1,b_1,\dots,b_r,1,\dots,1,b_1,\dots,b_r]$, where $k-1$ is the number of inserted $1$'s. We start by analysing the coefficients of the new continued fraction. Indeed, we obtain the following:
\begin{itemize}
    \item For every $2\leq i<r$, and $1\leq j<k$ we have that $c_i=c_{i+jr}=b_i$.
    \item For every $1\leq j<k$ we have that $c_{1+jr}=c_1-1=b_1-1$.
    
    \item For every $1\leq j<k$ we have that $c_{r+jr}=c_{kr}-1=b_r-1$.
\end{itemize}

In particular, if we fix a divisor $1<u<kr$ of $k$ then $c_{1+jvr}=c_{1}-1$, and $c_{r+jvr}=c_r-1$ for every $1\leq j<u$. Also, we have that for every $2\leq i<vr$, and $1\leq j<u$ we have that $c_{i}=c_{i+jvr}$. Thus, we can also obtain $[c_1,\ldots,c_{kr}]$ from a core $[a_1,\ldots,a_{vr}]$ by inserting $u-1$ $1$'s, where $k=uv$. However, it may not be valid if we choose a divisor $u$ of $r$.

Now, let $[e_1,\dots,e_s]$ be a core. Assume that $s$ is not a prime number. Then, we have that the core $[e_1,\ldots,e_s]$ is not minimal if and only if it fulfills the conditions above for some divisor $u>1$ of $s$.

Say $s$ is prime. By the conditions shown above, we have that $[e_1,\ldots,e_s]$ is a minimal core if and only if $[e_1,\ldots,e_s]=[e_1,e_1-1,\dots,e_1-1,e_1]$ (it is obtained by starting in $[e_1+1]$).

\end{proof}

\begin{lemma}\label{T generalizadas E.C=1} Let $W$ be a stable surface with a unique generalized T-singularity $[a_1,\dots,a_s]$ of center $[b_1,\dots,b_{r}]$. Assume that the minimal model of the minimal resolution of $W$ has canonical class nef. Suppose that the maximal exceptional divisor $E_i$ has diagram of type $(i)$, and there is not a $(-1)$-curve intersecting the ends of $C$, then

  \begin{equation*}
    2\delta\leq \sum_{j=1}^{s}\big( a_j-2\big)-2.
\end{equation*}
\end{lemma}
\begin{proof} Let $\Gamma$ be the curve in $C$ which intersects the $(-1)$-curve in $E_i$. We have that $[a_1,\dots ,a_s]$ is of the following form:
$$[2,\dots ,2,x_1,\dots ,x_{s-l-1},x_s+l]$$
where $l$ is the number of $2$'s on the left side, and $x_s+l\geq 3$. Note that $x_1\geq 3$ because of the admissibility of $[b_1,\dots,b_{r}]$.

The curve $\Gamma$ cannot be a $(-2)$-curve on the left of the chain, because $E_i$ does not have loops. Thus $\Gamma$ is a curve $C_{l+j}$ such that $C_{l+j}^2=-x_j$ for some $1\leq j\leq s-l-1$.

By Remark \ref{T-sum}, we have 

$$\sum_{j=1}^{s}\big( a_j-2\big)=\sum_{j=1}^{s-l-1}\big( x_j-2\big)+l.$$

Let us suppose that $\Gamma=C_{l+1}$, that is $\Gamma^2=-x_1$. After contracting the curves $F,C_1,\dots ,C_{l}$, the curve $\Gamma$ becomes a curve $D$ which has self-intersection equal to $-x_1+l+2$. By the adjunction formula and because $K_S$ is nef, we obtain that $x_1-2\geq l+2$. Due to the fact that $0\leq x_j-2$ for all $j$, we obtain

$$l+2\leq x_1-2\leq \sum_{j=1}^{s}\big( a_j-2\big)-l,$$
so $2\delta \leq \sum_{j=1}^{s}\big( a_j-2\big)-2$ because by Theorem \ref{deltas} we have that $\delta=l$. 

On the other hand, if $\Gamma=C_j$ for $1<j\leq s-l-1$, we obtain that the curve $C_j$ becomes a curve $D$, which has $D^2=-x_j+l+1$. Because $K_S$ is nef, we obtain that $x_j\geq l+3$. So we have

$$l+3-2+1\leq l+3-2+x_1-2\leq x_j-2+x_1-2\leq \sum_{j=1}^{s}\big( a_j-2\big)-l,$$
so we obtain that $2\delta \leq \sum_{j=1}^{s}\big( a_j-2\big)-2$.
\end{proof}


\begin{notation}\label{notation B} We denote by $\mathcal{B}([b_1,\dots,b_r])$ to the set formed for each iteration of $(ii)$ (the T-chain algorithm) applied to $[b_1,\dots,b_r]$. (See Definition \ref{generalized T-singularity}).
\end{notation}

\begin{remark}\label{T-sum} Let $[b_1,\ldots,b_r]$ be a continued fraction. Let $[a_1,\dots ,a_{s}]$ be an element of $\mathcal{B}([b_1,\dots ,b_{r}])$. Then, by a direct computation, we obtain
\begin{equation}\label{temp13}
    \sum_{j=1}^{s}(a_j-2)=\sum_{j=1}^{r}(b_j-2)+(s-r).
\end{equation}

Assume that $[b_1,\ldots,b_r]$ is a core. Let us denoted by $[b_1^k,\dots, b_{r_k}^k]$ the resulting continued fraction 
$[b_1,\dots,b_r,1,b_1,\dots,b_r,1,\dots,1,b_1,\dots,b_r]$, where $k$ is the number of inserted $1$'s. Then, $r_k=r(k+1)$ and

 \begin{equation}\label{temp16}
    \sum_{j=1}^{r_{k}}(b_j^k-2)= (k+1)\sum_{j=1}^{r}(b_j-2)-2k.
\end{equation}

Therefore, given a generalized T-singularity $[a_1,\dots, a_s]$ of center $[b_1,\ldots,b_r]$, we have
 \begin{equation}\label{temp11}
    \sum_{j=1}^{s}(a_j-2)= (k+1)\sum_{j=1}^{r}(b_j-2)-2k+(s-r_k),
\end{equation}
where $k$ is a number such that $[a_1,\dots, a_s]$ belongs to $\mathcal{B}([b_1^k,\dots, b_{r_k}^k])$.
\end{remark}



\begin{lemma}\label{previous 1.9} Let $W$ be a stable surface with only one generalized T-singularity with a fixed center $[b_1,\ldots,b_r]$, say at $P\in W$. Suppose that the minimal model $S$ of the minimal resolution of $W$ has canonical class nef. Assume that $K_W^2<c$ for some positive number $c$. Then, one of the following holds

\begin{itemize}

    \item[(i)] Assume that there is not a $(-1)$-curve intersecting the ends of the chain that resolves $P$, or that we have $E_i\cdot C\geq 2$ for every exceptional divisor, then

\begin{equation*}
\sum_{j=1}^{s}\big( a_j-2\big)< 4c+6, \textnormal{ and }s< 15c+40.
\end{equation*}
    
    \item[(ii)] Assume that there exists a $(-1)$-curve intersecting the ends of the chain that resolves $P$, then $P\in \mathcal{B}([b_1^u,\ldots,b_{r_u}^u])$ for some $u<2c+1$. Moreover, we have that there exists a non-negative number $m'$ such that $m'+1\leq \sum_{j=1}^{r_u}(b_j^u-2)$, and

    \begin{equation*}
    K_{W}^2=K_{S}^2+\sum_{j=1}^{r_u}(b_j^u-2)-(m'+1)-\bigg(\dfrac{2(n-1)-q-q'}{n}\bigg).
    \end{equation*}
\end{itemize}
\end{lemma}

\begin{proof}

We start by fixing some notation. Let $\frac{n}{q}=[a_1,\dots, a_s]$ be the Hirzebruch-Jung continued fraction associated to $P \in W$. We first note that it can be assumed that $[b_1,\ldots,b_r]$ is a core. Otherwise, by Theorem \ref{admissibility}, there exists a core $[e_1,\ldots,e_s]$ such that $[b_1,\ldots,b_r]$ belongs to $\mathcal{B}({[e_1,\ldots,e_s]})$, and that $[a_1,\dots, a_s]$ is a generalized T-singularity of center $[e_1,\ldots,e_s]$.

Now, let $\phi\colon X\to W$ be the minimal resolution of $P$ and let $C$ be the chain of exceptional rational curves. Let $\pi\colon X\to S$ be a birational morphism to the minimal model $S$.  By Corollary \ref{descartar tipo iii} we have the following two cases.

\textbf{(1)} We have that $E_i\cdot C\geq 2$ for every exceptional divisor of $\pi$. That is, $\delta=0$. In this case, by Theorem \ref{deltas} we have
\begin{equation}\label{Eqdelta}
\sum_{j=1}^{s}\big( a_j-2\big)< 2c+4, \textnormal{ and }s< 13c+38.
\end{equation}

\textbf{(2)} There is a unique exceptional divisor $E_i$ of $\pi$ such that $E_i\cdot C=1$, and its graph is maximal. By Corollary \ref{descartar tipo iii}, we also have that $\Gamma_{E_i}$ is of type $(i),$ $(ii)$ or $(iv)$. Now, we divide this case into the following sub-cases.

\textbf{(2.A)} Suppose that $E_i$ has a diagram of type $(i)$, and there is not a $(-1)$-curve intersecting the chain $C$ at both ends. Then, by putting together the bound for $\delta$ shown in Lemma \ref{T generalizadas E.C=1}, and Theorem \ref{deltas} we obtain that
\begin{equation}\label{EqA}
 \sum_{j=1}^{s}(a_j-2)<4c+6,\textnormal{ and }s< 15c+40.    
\end{equation}

\textbf{(2.B)} Assume that $E_i$ has a diagram of type $(ii)$. Then, by Theorem \ref{deltas} we obtain that $\delta=1$, and
\begin{equation}\label{EqB}
 \sum_{j=1}^{s}(a_j-2)<2c+5,\textnormal{ and }s< 13c+39.    
\end{equation}

By putting cases (1), (2.A), and (2.B) together, we obtain the first part of the statement. By Corollary \ref{descartar tipo iii}, the last case is the following.

\textbf{(2.C)} Suppose that $E_i$ has a diagram of type $(iv)$ or $E_i$ has a diagram of type $(i)$, and there exist $(-1)$-curve intersecting both ends of $C$. Let $F$ be the $(-1)$-curve of $E_i$. 

\begin{claim} Let $P$ be a T-singularity (a generalized T-singularities of center $[4]$). Then there is not a $(-1)$-curve intersecting both ends of $C$.
\end{claim}

Indeed, we know that a T-singularity $P$ can be expressed as $\frac{1}{dn^2}(1,dna-1)$ for some natural numbers $n,d,a$ such that $gcd(n,a)=1$, and $d$ is a square-free. Let us suppose that there exists a $(-1)$-curve $F$ which intersects both ends of $C$. Then, we obtain that 

$$\phi(F)\cdot K_W=-1+1-\frac{dna-1+1}{dn^2}+1-\frac{dn(n-a)-1+1}{dn^2}=0,$$
since the discrepancies of the ends of the chain are $-1+\frac{dna-1+1}{dn^2}$ and $-1+\frac{dn(n-a)-1+1}{dn^2}$. (See e.g. \cite[Section 2.1]{urzua2016identifying}). But, that violates the condition of being ample for $K_{W}$. This completes the proof of Claim.

Therefore, we obtain that $[a_1,\ldots,a_s]$ cannot be a T-singularity (an usual T-singularity). Now, we know that $[a_1,\dots,a_s]\in\mathcal{B}([b_1^u,\dots,b_{r_u}^u])$ for some $u\geq 0$. (See Notation \ref{reduced-notation}). Here, by Definition \ref{def core} we have that $[b_1^u,\dots,b_{r_u}^u]$ is a core, and then $r_u=r(u+1)$ (see Remark \ref{T-sum}). 

We claim that $u<2c+1$. Indeed, we know that $b_1^u=b_1>2$, and $b_{r_u}^u=b_r>2$, because $[b_1,\ldots,b_r]$ is a core. Also, by the formation rule of $[a_1,\dots, a_s]$, we obtain that $E_i$ contains exactly $(s-r_u)$ curves in $C$. They are contracted by starting at $F$. Thus, we know that $\delta=s-r_u$. (see Cases (B.1), and (C.2) in Theorem \ref{deltas}). So, by plugging the formula in \eqref{temp11}  into the inequality \eqref{sumc}, we obtain that  
\begin{equation}\label{temp7}
   (u+1)\sum_{j=1}^{r}(b_j-2)-2u+(s-r_u)<2c+4+ s-r_u.
\end{equation}

Due to the fact that $[a_1,\ldots,a_s]$ is not a T-singularity, we obtain that neither is $[b_1,\ldots,b_r]$. So, we have that $\sum_{i=1}^{r}(b_j-2)\geq 3$. Thus, by \eqref{temp7} we conclude that $0\leq u<2c+1$. Therefore, we have proved that 

\begin{equation*}
P\in \bigcup_{u=0}^{\lfloor 2c+1 \rfloor}\mathcal{B}([b_1^u,\dots,b_{r_u}^u]),
\end{equation*}
and so $P\in \mathcal{B}([b_1^u,\ldots,b_{r_u}^u])$ for some $u<2c+1$.

In addition, by \eqref{temp16} in Remark \ref{T-sum}, we obtain that $r_u<r(2c+2)$ and

 \begin{equation}\label{temp17}
    \sum_{j=1}^{r_{u}}(b_j^u-2)< (2c+2)\sum_{j=1}^{r}(b_j-2).
\end{equation}

On the other hand, by \eqref{temp13} in Remark \ref{T-sum}, we know that
\begin{equation}\label{temp18}
    \sum_{j=1}^{s}(a_j-2)=\sum_{j=1}^{r_u}(b_j^u-2)+(s-r_u).
\end{equation}

Let $m$ be the number of blow downs necessary to reach the minimal model $S$ from $X$. Note that by the formation rule of $[a_1,\dots, a_s]$, we can write $m=(s-r_u+1)+m'$ with $m'\geq 0$. By putting \eqref{temp18} in Equation (\ref{canonical divisor formula WS}), we obtain that
\begin{equation}
K_{W}^2=K_{S}^2+\sum_{j=1}^{r_u}(b_j^u-2)-(m'+1)-\bigg(\dfrac{2(n-1)-q-q'}{n}\bigg).    
\end{equation}
Note that by Lemma \ref{E.C}, and Remark \ref{FC>2} for $m$, we obtain $m'+1 \leq \sum_{j=1}^{r_u}(b_j^u-2)$. 
\end{proof}

\begin{definition}\label{property (*)} Let $\{W_k\}$ be a sequence of stable surfaces with only one generalized T-singularity of center $[b_1,\ldots,b_r]$. We say that $\{K_{W_k}^2\}$ satisfy the property (*) if there exists an infinite set of indices $J$ such that
\begin{itemize}

\item The self-intersection $K_{S_k}^2$ is constant for every $k \in J$.

\item There exists a $(-1)$-curve intersecting the ends of the chain that resolves $P_k$ for every $k\in J$.
    
\item There exists a number $u\geq 0$ such that $P_k\in\mathcal{B}([b_1^u,\ldots,b_{r_u}^u])$ for every $k\in J$.

\item The reduced Hirzebruch-Jung continued fraction of $P_k$ is different for each $k\in J$.
\end{itemize}
\end{definition}

\begin{proof}[Proof of Theorem \ref{A1}] We start by fixing some notation. Let $\frac{n_k}{q_k}=[a_1,\dots, a_s]$ be the Hirzebruch-Jung continued fraction associated to $P_k \in W_k$. As in the proof of Lemma \ref{previous 1.9}, we can assume without loss of generality that $[b_1,\ldots,b_r]$ is a core. Let $\phi_k\colon X_k\to W_k$ be the minimal resolution of $P_k$ and let $C^k$ be the chain of exceptional rational curves. Let $\pi_k\colon X_k\to S_k$ be a birational morphism to the minimal model $S_k$. Let $c$ be a positive but arbitrary real number.

Assume that $\{K_{W_k}^2\}$ has accumulation points. Then, there exists a positive number $c$ such that $\{K_{W_k}^2:K_{W_k}^2< c\}$ has accumulation points. Let $J'$ be the set of indices $k$ such that $K_{W_k}^2< c$, and there exists a $(-1)$-curve intersecting both ends of $C^k$. We know that $J'$ is an infinite set. Otherwise, by Lemma \ref{previous 1.9} we obtain bounds for $\sum_{j=1}^{s}(a_j-2)$, and $s$ which only depend on $c$. So, we would have that $\{K_{W_k}^2:K_{W_k}^2< c\}$ has no accumulation points. But this is impossible, so $J'$ is an infinite set of indices. More precisely, we know that

\begin{equation}
  \textnormal{Acc}\big(\{K_{W_k}^2:K_{W_k}^2< c\}\big)=\textnormal{Acc}\big(\{K_{W_k}^2:k\in J'\}\big).
\end{equation}

Again, by Lemma \ref{previous 1.9} for each $k\in J'$ there exists $u<2c+1$ such that $P_k\in \mathcal{B}([b_1^u,\ldots,b_{r_u}^u])$, and so

\begin{equation}\label{EqC}
K_{W_k}^2=K_{S_k}^2+\sum_{j=1}^{r_u}(b_j^u-2)-(m_k'+1)-\bigg(\dfrac{2(n_k-1)-q_k-q_k'}{n_k}\bigg),
\end{equation}
where $0<m_k'+1 \leq \sum_{j=1}^{r_u}(b_j^u-2)$.

So, by replacing the bound for $m_k'$ in \eqref{EqC}, it follows that $K_{S_k}^2<c+2$ for every $k\in J'$. Thus, because $K_{S_k}^2$ is an integer for every $k$, we obtain that $\{K_{S_k}^2:k\in J'\}$ is a finite set. 



For each $u<2c+1$, let $J_u'\subseteq J'$ be the set of indices $k$ such that $P_k\in \mathcal{B}([b_1^u,\ldots,b_{r_u}^u])$. By Lemma \ref{previous 1.9}, we know that $J'=\bigcup_{u=0}^{\lfloor 2c+1\rfloor} J_{u}'$. So, there exists at least one $u<2c+1$ such that $J_u'$ is an infinite set of indices. Note that between the infinite sets $J'_u$ we can choose one of them with the property that the reduced Hirzebruch-Jung continued fraction of $P_k$ are different for each $k\in J'_u$ (except maybe for a finite set of $J'_u$). On the contrary, we would have by \eqref{EqC} that $\{K_{W_k}^2:k\in J'_u\}$ is a finite set for every $u$, and then $\{K_{W_k}^2:K_{W_k}^2<c\}$ would not have accumulation points. But this is not possible. Let $J_{u_0}'$ such a set. Now, because we have that $\{K_{S_k}^2:k\in J'\}$ is a finite set then we may choose an infinite subset of indices $J\subseteq J_{u_0}'$ such that $K_{S_k}^2$ is constant for every $k\in J$.

Then, we know that $J$ is the set with the desired properties of the statement in Theorem \ref{A1}.

Conversely, let us suppose that there exists an infinite set of indices $J$ such that
\begin{itemize}

\item We have that $K_{S_k}^2$ is constant for every $k\in J$.

\item There exists a $(-1)$-curve intersecting the ends of the chain that resolves $P_k$ for every $k\in J$.
    
\item There exists a number $u\geq 0$ such that $P_k\in\mathcal{B}([b_1^u,\ldots,b_{r_u}^u])$ for every $k\in J$.

\item The reduced Hirzebruch-Jung continued fraction of $P_k$ is different for each $k\in J$.
\end{itemize}

By using those statements and Lemma \ref{previous 1.9}, it follows that for every $k\in J$
\begin{equation}\label{EqC3}
K_{W_k}^2=K_{S_k}^2+\sum_{j=1}^{r_u}(b_j^u-2)-(m_k'+1)-\bigg(\dfrac{2(n_k-1)-q_k-q_k'}{n_k}\bigg),
\end{equation}
where $0<m_k'+1 \leq \sum_{j=1}^{r_u}(b_j^u-2)$. Let $c'$ be a positive number such that $K_{S_k}^2=c'$ for every $k\in J$. Then, by \eqref{EqC3} we obtain that $$K_{W_k}^2<c'+\sum_{j=1}^{r_u}(b_j^u-2)+2,$$ 
for every $k\in J$. So, we have that $\{K_{W_k}^2:k\in J\}$ is a bounded set.

Now, we construct an infinite set of indices $J'$ such that the continued fraction of $P_{k_{i+1}}$ is obtained by applying the T-chain algorithm (see Definition \ref{generalized T-singularity}) to the continued fraction of $P_{k_i}$ for every $k_i\in J'$. In fact, let us fix an integer $s\geq r$. By using the formation rule in $\mathcal{B}([b_1^u,\ldots,b_{r_u}^u])$ and the fact that $P_k$ has different continued fraction for every $k\in J$, we know that there exist finitely many $k\in J$ such that $P_k$ has a continued fraction of length $s$. Thus, we can choose $k_0\in J$ such that the continued fraction of $P_k$ is obtained by applying the T-chain algorithm to the continued fraction of $P_{k_0}$ for infinitely many $k\in J$. Let $J_{k_0}\subseteq J$ be a subset of indices with such a property. In the same way, we can choose an index $k_1 \in J_{k_0}$ such that the continued fraction length of $P_{k_1}$ is greater than the length of $P_{k_0}$, and that $P_k$ is obtained by applying the T-chain algorithm to the continued fraction of $P_{k_1}$ for infinitely many $k\in J$. Let $J_{k_1}\subseteq J_{k_0}$ be an infinite set of indices with such a property. By using an inductive argument, we may construct an infinite set of indices $J'=\{k_0,k_1,\ldots\}$ with the desired property. 

Observe that the quotients $2(n_{k_i}-1)-q_{k_i}-q_{k_i}'/n_{k_i}$ are different for each $k_i\in J'$. Indeed, by Proposition \ref{formation rule}, one can compute directly that 
\begin{equation*}
    \dfrac{2(n_{k_{i+1}}-1)-q_{k_{i+1}}-q_{k_{i+1}}'}{n_{k_{i+1}}}<\dfrac{2(n_{k_i}-1)-q_{k_i}-q_{k_i}'}{n_{k_i}}.
\end{equation*}

Therefore, by \eqref{EqC3} we know that $\{K_{W_k}^2:k\in J\}$ is an infinite set which also is bounded. So, we conclude that $\{K_{W_k}^2\}$ has accumulation points.

\end{proof}

\begin{proposition}\label{A2} Let $W$ be a stable surface which has only one generalized T-singularity $P\in W$ with continued fraction $[b_1,\dots,b_{r}]$. Assume that there exists a $(-1)$-curve intersecting both ends of the chain $C$ associated to $P$. Then there exist a sequence $\{W_k\}$ of stable surfaces with only one generalized T-singularity $P_k$ of center $[b_1,\ldots,b_r]$ such that $\{K_{W_k}^2\}$ has an accumulation point.
\end{proposition}

\begin{proof} Let us write $W_1:=W$, and let $X_1$ be the minimal resolution of $P$. Let $F_1$ be the $(-1)$-curve intersecting the ends $C_1$, and $C_r$ of $C$.

\textbf{Step 1.} Let $X_2$ be the smooth surface obtained by blowing up at the point in $F_1\cap C_1$. Let $F_2$ be the exceptional curve in $X_2$. Note that $X_2$ has a configuration of $(r+1)$-rational curves (with SNC), and that $F_2$ intersects the ends of the chain. The new configuration has continued fraction $[b_1+1,\dots,b_{r-1},b_r,2]$.

\textbf{Step k.} Let us assume constructed the surfaces $X_1,\dots,X_{k-1}$, inductively. Let $F_{k-1}$ be the exceptional curve in $X_{k-1}$. In the same way of Step 1, we construct a smooth surface $X_k$ which is obtained by blowing up at the point in $F_{k-1}\cap C_{1}$, where $C_{1}^2=-(b_1+k-2)$. Here, we obtain a new configuration of rational curves with SNC, and that $F_k$ (the exceptional curve in $X_{k+1}$) intersects the ends of the chain. That configuration has continued fraction $[b_1+(k-1),b_2\dots,b_r,2,\dots,2]$, where $k-1$ is the number of $2$'s on the right side. Thus, we obtain a sequence $\{X_k\}$ of smooth surfaces with a configuration $[b_1+(k-1),b_2\dots,b_r,2,\dots,2]$, and a $(-1)$-curve intersecting the ends of the chain.

Now, by Artin's contractibility Theorem \cite[Thm. 2.3]{artin1962some}, we may contract the configuration in $X_k$ for every $k$, to obtain a normal projective surface $W_k$ with only one cyclic quotient singularity $P_k\in \mathcal{B}_{[b_1,\dots, b_{r}]}$. In particular, $P_k$ is a generalized T-singularity of center $[b_1,\dots,b_r]$. Thus, to complete the proof, we must show that $K_{W_k}$ is ample divisor for every $k$. In fact, let $a_j$ be the discrepancies of $[b_1,\ldots b_r]$ for $j=1,\ldots,r$, and let $a_j'$ be the discrepancies of $[b_1+1,\ldots,b_r,2]$ for $i=1,\ldots,r+1$. We recall that $-1<a_j,a_j'\leq 0$ for every $j$.

\begin{claim}\label{discrepancies} We have that $a_j>a_j'$ for every $j=1,\ldots,r$.
\end{claim}
Indeed, let $c_j,d_j$ (respectively $c_j',d_j'$) be the auxiliary coefficients defined in Remark \ref{triangular matrix} for $[b_1,\ldots,b_r]$ (respectively for $[b_1+1,\ldots,b_r,2]$ ). We recall that $c_1=-1/b_1$, $c_1'=-1/(b_1+1)$, $d_1=(2-b_1)/b_1$, and $d_1'=(1-b_1)/(b_1+1)$.

We have that $-1<c_j'<c_j<0$ for every $j=1,\ldots r-1$. Indeed, by a direct computation we know that $-1<c_1'<c_1<0$. Let us suppose the statement for $j-1$. Then,  
 \begin{equation*}
   c_j'=\dfrac{-1}{b_j+c_{j-1}'}<\dfrac{-1}{b_j+c_{j-1}}=c_j,  
 \end{equation*}  
one can check directly that $-1<c_j,c_j'<0$ by using that $b_j\geq 2$.

Also, we obtain that $d_j'<d_j$ for every $j=1,\ldots,s$. In fact, we note that $d_1-d_1'=2/(b_1(b_1+1))>0$. Assume the statement for $j-1$. Then,  
\begin{equation*}
   d_{j}'-d_{j}\leq (b_j-2-d_{j-1}')c_j'-(b_j-2-d_{j})c_j'=c_j'(d_{j-1}-d_{j-1}'), 
\end{equation*}
by using $c_j'<0$, and $d_{j-1}'<d_{j-1}$, we obtain that $d_j'<d_j$.

In addition, we note that because $-1<c_s'<0$ then

$$a_r-a_r'=d_r-d_r'(1-(c_r')^2)>d_r-d_r'>0.$$

Now, if we suppose that $a_{j+1}>a_{j+1}'$ then by using $c_j'<c_j<0$ we obtain that

$$a_j-a_j'=(d_j-d_j')+(c_j'a_{j+1}'-c_ja_{j+1})>0.$$

Thus, we obtain that $a_j>a_j'$ for every $j=1,\ldots,r$. This completes the proof of Claim. 

We will use Claim \ref{discrepancies} to prove the ampleness. Let $B_j$ the curves in the configuration on $X_2$, that is $B_j^2=-b_j$, $B_1^2=-(b_1+1)$, and $B_{r+1}^2=-2$. Let $f\colon X_2\to W_1$ be the map which contracts $F_2,B_1,\ldots,B_r$, and let $\phi\colon X_1\to W_2$ be the map which contracts $B_1,\ldots,B_{r+1}$ (the minimal resolution of $W_2$). Then, one can check that

\begin{equation*}
    \phi^*(K_{W_2})=f^*(K_{W_1})+(a_1+1)F_2+\sum_{j=1}^{r}(a_j-a_j')B_j-a_{r+1}'B_{r+1}.
\end{equation*}

By Claim \ref{discrepancies}, we obtain $K_{W_2}$ written  as an effective sum of divisors. So, we only need to check that $K_{W_2}^2>0$,  $\phi^*(K_2)\cdot F_2>0$ to prove the ampleness of $K_{W_2}$. We first recall that $-1-a_1-a_r=1-(2+q+q')/n$, where $[b_1,\ldots,b_r]=n/q$. (See e.g. \cite[Section 2.1]{urzua2016identifying}). 

Because of the ampleness of $K_{W_1}^2$, we obtain that $0<-1-a_1-a_r$. So, we have that $2+q+q'<n$. Let $N,Q$ be the integers such that $[b_1+1,\ldots,b_r,2]=N/Q$. By Proposition \ref{formation rule} we know that $N=2q-m+2n-q'$, $Q=2q-m$, and $Q'=q+n$. So,

\begin{equation}\label{temp8}
    \phi^*(K_2)\cdot F_2=-1-a_1'-a_{r+1}'=\dfrac{n(n-(2+q+q'))}{(n+q)(2n-q')+1},
\end{equation}
then, we obtain directly from \eqref{temp8} that $\phi^*(K_2)\cdot F_2>0$. 

On the other hand, we have that
\begin{equation}\label{temp9}
    K_{W_2}^2=K_{W_1}^2+\dfrac{2+Q+Q'}{N}-\dfrac{2+q+q'}{n},
\end{equation}
and
\begin{equation*}
    (2+Q+Q')n-(2+q+q')N= \dfrac{(n-(2+q+q'))(n^2+2nq-nq'-qq'+1)}{n}>0.
\end{equation*}

Thus, we obtain from \eqref{temp9} that $K_{W_2}^2>0$. Then, by the Nakai-Moishezon criterion we obtain that $K_{W_2}$ is ample. Note that we only use the facts that $K_{W_1}$ is ample, and the formation rule of the configuration in $X_2$ to prove that $K_{W_2}$ is ample. So, we also proved that $K_{W_k}$ ample implies $K_{W_{k+1}}$ ample for every $k$. Therefore, we have that $\{K_{W_k}^2\}$ has accumulation points.
\end{proof}

\begin{remark}\label{data} We recall some useful data from the proof of Proposition \ref{A2}. Let $W_1$ be a stable surface which has only one generalized T-singularity $P_1\in W_1$ with continued fraction $[b_1,\ldots,b_r]$. Let $X_1$ be the minimal resolution of $P_1$. Let $X_2$ be the surface obtained by blowing up $X_1$ as described in Step 1 (see proof of Proposition \ref{A2}). As we saw in the proof, we have that $W_2$ has a generalized T-singularity $P_2$ with the continued fraction associated $[b_1+1,\ldots,b_r,2]$, and such that $K_{W_2}$ is an ample divisor.   
\end{remark}

For the following proposition, let us consider the diagram
\begin{figure}[H]
    \centering

\begin{equation}\label{Diagram1}
\begin{split}
    \xymatrix{&X'_k\ar[d]_{\pi_k}\ar[rd]^{\phi_k}&\\ &\ar[ld]_{\pi}X'_{k_1}\ar[rd]^{\phi}&W'_k\\
S_{k_1}&&W'_{k_1}}
\end{split}  
\end{equation}

\end{figure}
where the birational morphism $\phi\colon X'_{k_1}\to W'_{k_1}$ (Respectively $\phi_{k}\colon X'_{k}\to W'_{k}$) is the minimal resolution of $W'_{k_1}$ (Respectively $W'_k$), and the surface $S_{k_1}$ is the minimal model of $X'_{k_1},X'_{k}$.


\begin{proposition}\label{A3} Under the assumptions of Theorem \ref{A1}, let $\nu_{\infty}\in \mathbb{Q}$ be an accumulation point of $\{K_{W_k}^2\}$. Then, there exists a sequence $\{W'_k\}$ of stable surfaces such that 

\begin{itemize}
\item $W'_k$ has only one generalized T-singularity $P_k$ which is analytically the same singularity of $W_k$ for every $k\in I$, where $I$ is an infinite set of indices.

\item There exist $k_1\in I$ such that for every $k\in I$, the minimal resolution $X'_{k}$ of $W'_{k}$ is obtained by blowing up the minimal resolution $X'_{k_1}$ of $W'_{k_1}$. (See Diagram \ref{Diagram1}).

\item The limit of the sequence $K_{W_k'}^2$ is $\nu_{\infty}$.
\end{itemize}
\end{proposition}

\begin{proof} Let $J$ be an infinite subset of indices as in Theorem \ref{A1} such that $\nu_{\infty} \in \textnormal{Acc}(\{K^2_{W_k}:k\in J\})$. Let us choose an infinite subset $J'$ of $J$ such that the subsequence $\{K_{W_k}\}_{k\in J'}$ converges to $\nu_{\infty}$. As we saw in the proof of Theorem \ref{A1}, we can choose an infinite set of indices $J''\subseteq J'$ such that $K_{W_{k_i}}^2$ goes to $\nu_{\infty}$ when $i$ goes to infinity, and also such that the continued fraction of $P_{k_i}$ is obtained by applying the T-chain algorithm to the continued fraction of $P_{k_{i-1}}$.

By using Lemma \ref{previous 1.9}, it follows that for every $k\in J''$
\begin{equation}\label{temp19}
K_{W_k}^2=K_{S_k}^2+\sum_{j=1}^{r_u}(b_j^u-2)-(m_k'+1)-\bigg(\dfrac{2(n_k-1)-q_k-q_k'}{n_k}\bigg),
\end{equation}
where $0<m_k'+1 \leq \sum_{j=1}^{r_u}(b_j^u-2)$. So, we may choose an infinite set of indices $I\subseteq J''$ such that $m_k'$ is constant for every $k\in I$. After renaming the surfaces $W_k$ we may suppose that $I=\mathbb{N}$. 

Let us write $W_1':=W_1$, and let $X_1':=X_1$ be the minimal resolution of $P_1$. Now, let $X_2'$ be the surface obtained by blowing up the configuration $C_1$ associated to $P_1$ such that after contracting the new configuration $C_2$ in $X_2'$, we obtain a normal projective surface $W_2'$ with the generalized T-singularity $P_2$. We remark that is possible because the continued fraction of $P_2$ is obtained by applying the T-chain algorithm to the continued fraction of $P_1$. By using Remark \ref{data} (maybe several times) we obtain that $K_{W_2}$ is an ample divisor. Also, by using the facts that $K_{S_k}^2,m_k'$ are constants for every $k\in I$ in \eqref{temp19}, we obtain that $K_{W'_2}^2=K_{W_2}^2$. 

Finally, by using an inductive argument, we construct a sequence of stable surfaces $\{W_k'\}$ with the desired properties of the statement. 
\end{proof}





To describe the behavior of the accumulation points for stable surfaces with one generalized T-singularity with a fixed center (Theorem \ref{A1}), we used that the canonical class of $S$ is nef. We do not know what happens otherwise. We note that Theorem \ref{deltas} is still valid for $K_S$ not nef, and so it could be used for some further analysis. Also, we are interested in finding properties for generalized T-singularities, like the one that motives the definition of T-singularities in \cite{kollar1988threefolds}.

On the other hand, the general question on how accumulation points show up for stable surfaces with only one cyclic quotient singularity remains open. 

\begin{remark} Given a sequence as in Theorem \ref{A1}, we saw in the proof of Theorem \ref{A1} that every accumulation point of a sequence $\{K_{W_k}^2\}$ can be obtained from a subsequence such that every $W_k$ has only one singularity in the set $\mathcal{B}([b_1^u,\ldots,b_{r_u}^u])$ for a fixed $u\geq 0$. In that case, we recall that $K_{W_k}^2$ are related by the following formula
\begin{equation}\label{quotients}
K_{W_k}^2=c+\sum_{j=1}^{r_u}(b_j^u-2)-(m+1)-\bigg(\dfrac{2(n_k-1)-q_k-q_k'}{n_k}\bigg),
\end{equation}
where $c,m$ are fixed numbers, and $c=K_{S_k}^2$. By Proposition \ref{formation rule}, we know a recursive way of computing the quotients in \eqref{quotients}. So, we have the following question
\end{remark}

\begin{question} Let $\{W_k\}$ be a sequence of stable surfaces as in Theorem \ref{A1}. What are the accumulation points of $\{K_{W_k}^2\}$?
\end{question}


We saw in Proposition \ref{A3} that every accumulation point of stable surfaces with only one generalized T-singularity can be constructed by blowing up a certain configuration of curves in a smooth surface and then contracting the new configuration obtained. Following that idea, we want to finish with the following questions concerning that topic.

\begin{question} Let $\{W_k\}$ be a sequence of stable surfaces such that any $W_k$ has only one cyclic quotient singularity, say at $P_k \in W_k$. Suppose that the minimal model of $W_k$ has canonical class nef. Let $\nu_{\infty}\in \textnormal{Acc}(\{K_{W_k}^2\})$. Then, there is a sequence $\{W'_k\}$ of stable surfaces and an infinite set of indices $I$ such that 

\begin{itemize}
\item $W'_k$ has only one cyclic quotient singularity $P_k$ which is analytically the same singularity of $W_k$ for every $k\in I$.

\item Let $E^k$ be an exceptional divisor in the minimal resolution of $W'_k$ such that $\Gamma_{E^k}$ is maximal (see Definition \ref{maximal graph}). Then, for every $k\in I$ we have that $E^k$ has only one type of diagram. Namely, a diagram of type $(i)$, $(iii)$ or $(iv)$ (see Definition \ref{Long diagram}).

\item There exist $k_1\in I$ such that for every $k\in I$, the minimal resolution $X'_{k}$ of $W'_{k}$ is obtained by blowing up the minimal resolution $X'_{k_1}$ of $W'_{k_1}$ .

\item The limit of the sequence $K_{W_k'}^2$ is $\nu_{\infty}$.
\end{itemize}
\end{question}

\begin{question} Let $\{W_k\}$ be a sequence of stable surfaces such that any $W_k$ has only one cyclic quotient singularity, say $P_k\in W_k$. Assume that singularities $P_k$ are analytically different for every $k$. Let $E^k$ be an exceptional divisor in the minimal resolution of $W_k$ such that $\Gamma_{E^k}$ is maximal (see Definition \ref{maximal graph}). Assume that $E_k$ has only one type of diagram. Namely, a diagram of type $(i)$, $(iii)$ or $(iv)$ (see Definition \ref{Long diagram}). Then, the set $\{K_{W_k}^2\}$ has accumulation points.
\end{question}



  


\bibliographystyle{alpha}
\bibliography{References}



       
\end{document}